\documentclass{article}

\usepackage{graphicx} % Required for inserting images
\usepackage[utf8]{inputenc}
\usepackage{float}
\usepackage{url}
\usepackage[margin =1in]{geometry}
\usepackage{amssymb,amsthm}
\usepackage{thm-restate}
\usepackage{amsmath}
\usepackage{enumerate}
\usepackage{subcaption}
\usepackage{algorithm2e} 
\LinesNumbered
\usepackage{cite}
\usepackage[noend]{algpseudocode}
\usepackage{breqn}
\usepackage{hyperref}
\usepackage{cleveref}

\newtheorem{theorem}{Theorem}

\hypersetup{
    colorlinks=true,
    linkcolor=blue,
    filecolor=magenta,      
    urlcolor=cyan,
    pdftitle={Overleaf Example},
    pdfpagemode=FullScreen,
    }
\usepackage{bbm}

\def\Reals{\mathbb{R}}
\def\one{\mathbbm{1}}

\def\cB{\mathcal{B}}
\def\cD{\mathcal{D}}
\def\cP{\mathcal{P}}
\def\cY{\mathcal{Y}}

\def\vec{\mathrm{vec}}
\def\mat{\mathrm{mat}}
\def\Egk{\mathbb{E}_{I_k}[\tilde{g}^k]}
\def\msperp{M_{S^\perp}}

\newtheorem{definition}{Definition}

\newtheorem{proposition}{Proposition}
\newtheorem{assumption}{Assumption}
\title{Avoiding Geometry Improvement in Derivative-Free Model-Based Methods via Randomization}
\author{Matt Menickelly}
\begin{document}

\maketitle

\begin{abstract}
    We present a technique for model-based derivative-free optimization called \emph{basis sketching}. 
    Basis sketching consists of taking random sketches of the Vandermonde matrix employed in constructing an interpolation model. 
    This randomization enables weakening the general requirement in model-based derivative-free methods that interpolation sets contain a full-dimensional set of affinely independent points in every iteration.
    Practically, this weakening provides a theoretically justified means of avoiding potentially expensive geometry improvement steps in many model-based derivative-free methods. 
    We demonstrate this practicality by extending the nonlinear least squares solver, \texttt{POUNDers} to a variant that employs basis sketching and we observe encouraging results on higher dimensional problems. 
\end{abstract}

\section{Introduction}
Derivative-free optimization (DFO) is an important and practical class of nonlinear optimization characterized by an assumption that derivatives of an objective function (and/or constraint functions) cannot be directly computed. 
Instead, it is assumed that one has access only to a black box oracle for computing the objective (and/or constraint) value(s). 
There exists an abundance of methods designed for derivative-free optimization, see e.g., the survey \cite{LMW2019AN}. 

In this paper, we focus on a particular subclass of these methods, in particular, model-based DFO trust-region methods for unconstrained optimization:
\begin{equation}
    \label{eq:unconstrained}
    \displaystyle\min_{x\in\Reals^n} f(x)
\end{equation}
for some $f(x):\Reals^n\to\Reals$. 
Such methods employ interpolation models of $f(x)$ as a proxy for a first- (or possibly second-) order Taylor model within a trust-region framework. 
While the convergence analysis of these methods and a formal analysis of interpolation model quality was largely pioneered by \cite{Conn1997b, Conn2006a, Conn2006b, Conn2009} and is most of the subject of the textbook \cite{Conn2009a},
the implementation and practical use of this subclass of methods was very much driven by Michael Powell, and forms the backbone of many of his well-known DFO solvers \cite{Powell2002a, Powell2003a, Powell2004a, Powell2004b, MJDP06, MJDP07b, Powell2009a, Powell2012convergence, Powell2013beyond}. 

It is well-known that methods within this subclass do not theoretically (or practically) scale well with the dimension $n$.
Intuitively,
in order to approximate a gradient $\nabla f(x)$, one must compute $n$ directional derivatives centered at $x$, which requires $\mathcal{O}(n)$ function evaluations. 
This linear dependence on $n$, barring some very restrictive assumptions, is inevitable; see
\cite{LMW2019AN}[Table 8.1] for a recent summary of worst case complexity results for methods in this subclass exhibiting this dependence on dimension. 

The intention of this paper is to suggest a framework for \emph{practically} alleviating this $\mathcal{O}(n)$ dependence. 
We do this by taking a recognizable framework (model-based trust-region methods) and \emph{randomizing} the construction of interpolation models by only updating an average model gradient and average model Hessian within a random (but realized from a judiciously selected distribution) subspace in each iteration. 
The inspiration for such a procedure draws heavily from \texttt{SEGA} \cite{hanzely2018sega}, which employs \emph{gradient sketches} to update an average gradient estimator. 
The name \texttt{SEGA} is intentionally close to \texttt{SAGA} \cite{DefazioBL14}, which can be viewed as a particular sketch of the summands of a finite sum, as opposed to a sketch of an $n$-dimensional gradient (see \cite{gower2021stochastic}).
Analogously, the contribution of this paper is an analogue of our recent work in \texttt{SAM-POUNDers} \cite{SAMP2022}, which randomly sketched a finite sum in a nonlinear least-squares problem by building on a model-based trust-region method for nonlinear least-squares problems, \texttt{POUNDers} \cite{SWCHAP14}. 

\subsection{Related Work}\label{sec:related}
Owing to practical relevance, there exists a body of literature concerning the scalability of DFO methods. 
Most notably, \cite{CRsubspace2021} analyzed a trust-region framework that iteratively selects a low-dimensional subspace, constructs an interpolation model \emph{only intended to be a reasonable approximation of an objective function on that subspace}, and then minimizes the low-dimensional model within a trust region.
By choosing the low-dimensional subspace according to constructions based on Johnson-Lindenstrauss transforms (see, e.g., \cite{element-JL}), \cite{CRsubspace2021} demonstrates a high probability worst case complexity result that is \emph{independent of $n$}. 
Such a result is encouraging in that it lends credence to methods based on iterative subspace selection, and motivates further exploration of methods such as the one in this paper. 
Prior work in subspace model-based methods (but without such rigorous convergence guarantees) includes \texttt{VXQR} \cite{Neumaier2011} and work with moving ridge functions \cite{gross2022optimization}. 

Other classes of DFO methods besides model-based methods have also considered issues of scalability. 
In particular, the class of direct search methods (i.e., methods based on making direct comparisons between pairs of function evaluations as opposed to constructing models, see \cite{AudetHare2017} for a textbook treatment) can intuitively always choose to evaluate fewer than $n$ poll points per iteration. 
Probabilistic convergence results for such randomized direct search methods are given in \cite{bergou2020stochastic, Gratton2015, Gratton2017direct}. 

We comment on the existence of additional DFO methods for large-scale problems outside of general model-based or direct search methods. 
There exist methods that make particular assumptions (and then exploit) some known structure of the problem, including partial separability \cite{Colson2005ops} or sparse Hessians \cite{Bandeira11}.
Others methods, based on a derivative-free method popularized by the analysis of Nesterov \cite{Nesterov2015}, attempt to only approximate one directional derivative per iteration; however, it is common knowledge among practitioners that these methods are practically very inefficient in terms of function evaluation complexity, see \cite{BCCS2021a} for a critical view.

Finally, there exists a large and growing body of related work in sketching in \emph{derivative-based} methods.
In the context of sketching for dimensionality reduction to yield subspace methods akin to those in \cite{CRsubspace2021}, 
sketching has been applied to Newton's method \cite{pilanci2017newton, gower2019rsn, berahas2020investigation}, quasi-Newton methods \cite{gower2016stochastic}, \texttt{SAGA} \cite{gower2021stochastic}, trust-region methods \cite{cartis2022randomiseda}, quadratic regularization methods \cite{cartis2022randomised}, and cubic regularization methods \cite{hanzely2020stochastic}.
For a distinct class of methods, (randomized) coordinate descent methods and their block variants employ directional derivative information in (randomized) coordinate directions.
While not subspace methods, or sketching methods, in exactly the terms we have used so far, coordinate descent methods effectively update a subset of variables parameterizing a coordinate-aligned subspace in each iteration. 
The structure of many problems in machine learning naturally lends itself to these methods. 
As a result, the literature on coordinate descent methods is vast, and so we only provide references to two good surveys of the topic, \cite{shi2016primer, wright2015coordinate}. 

\subsection{Motivation}
We are especially motivated by computationally expensive problems. 
In particular, and without any further quantification, we will assume that the cost (in time, energy consumption, dollar cost, or whatever relevant metric) of computing a single function evaluation $f(x)$ far exceeds the cost in the same metric of relevant linear algebraic procedures which are polynomial in $n$ (e.g., QR decompositions or solving trust region subproblems). 

We are additionally motivated by our recent work in \texttt{SAM-POUNDers} \cite{SAMP2022}.
\texttt{SAM-POUNDers} is motivated by problems in nuclear model calibration \cite{BMNORW2020}; 
these problems are naturally formulated as derivative-free (and potentially computationally expensive) nonlinear least squares problems of the form
\begin{equation}
    \label{eq:nonlsq}
    \displaystyle\min_{x\in\cD} \sum_{i=1}^p f_i(x)^2,
\end{equation}
where $\cD$ is defined by bound constraints. 
Unlike most assumptions in the literature, however, it is often the case in nuclear model calibration that most, if not all, of the individual expensive function evaluations $f_i(x)$ in \eqref{eq:nonlsq} can be performed in parallel. 
\texttt{SAM-POUNDers} exploits this by maintaining \emph{separate} local interpolation models of each $f_i(x)$, with each generally employing different interpolation sets. 
Most importantly, and like in a stochastic average gradient method (e.g. \texttt{SAG} \cite{Schmidt2013} or \texttt{SAGA} \cite{DefazioBL14}), not all of the $p$ local models of the function $f_i(x)$ in \eqref{eq:nonlsq}  will be updated in every iteration of \texttt{SAM-POUNDers}. 
Thus, in such a method, we must always have access to a (potentially stale) local model of each component function $f_i(x)$ that models $f_i(x)$ \emph{on the full $n$-dimensional space}.
As described in \Cref{sec:related}, a subspace method does not aim to provide such a model, electing instead to build a model of a function on a random subspace and then immediately discard the model for the next iteration. 
With this aim in mind, this paper will maintain a model -- in particular, a quadratic model defined by the \emph{average gradient} and \emph{average Hessian}) -- between iterations without discarding the model.
Moreover, having such an estimator of full-space gradient information can also aid in defining more practical stopping criteria, which is an identified shortcoming in subspace methods. 

\subsection{Contributions and Organization}
In this paper, we will present a method called \emph{basis sketching}.
Basis sketching maintains a running \emph{average estimator} of both the gradient and Hessian of $f(x)$, which is employed to compute a set of specifically and stochastically weighted estimators of the gradient and Hessian, called the \emph{ameliorated estimators}. 
While we will present a general framework for this method in \Cref{alg:bask}, we will specifically implement this algorithm using the model-building routines of \texttt{POUNDers} as a foundation. 
We will demonstrate through numerical results a clear advantage of the randomized variant presented in this paper over the original implementation of \texttt{POUNDers} on problems of sufficiently high dimension. 

We begin in \Cref{sec:interp} with preliminaries on the interpolation models that we will use in this work. 
Then, in \Cref{sec:estimators}, we will discuss the particular estimators we will employ that are derived from our interpolation models. 
In \Cref{sec:bask}, we will present the basis sketching algorithm, and will immediately afterwards discuss practical modifications in \Cref{sec:practical}. 
We conclude with a presentation of numerical results in \Cref{sec:numerical}, which are intended to show that basis sketching offers improvement over the standard (deterministic) geometry improvement and model improvement steps made in \texttt{POUNDers}. 

\section{Preliminaries on Interpolation Models}\label{sec:interp}
We begin with a discussion of general nonlinear interpolation models in \Cref{sec:gen_interp}, and we will then restrict ourselves to a discussion of underdetermined interpolation models in \Cref{sec:underdetermined}.  

\subsection{General Nonlinear Interpolation}\label{sec:gen_interp}
Denote by $\cP^d_n$ the space of polynomials of degree less than or equal to $d$ in $\Reals^n$.
We say that a polynomial $m(x)\in\cP^d_n$ \emph{interpolates} the function $f(x)$ at a point $y$ provided $m(y)=f(y)$. 
Suppose we have a set of points $Y = \{y^0,y^1,\dots,y^p\}\subset\Reals^n$. 
To find a polynomial $m(x)\in\cP^d_n$ that interpolates $f(x)$ at each point in $Y$, 
one can choose a basis for $\cP^d_n$, which we denote by $\Phi = \{\phi_0,\phi_1,\dots,\phi_q\}$, 
and solve for the coefficients $\alpha$ in the system

\begin{equation}
\label{eq:get_coeffs}
M(\Phi, Y)\alpha_\Phi = f(Y),
\end{equation}
where we have written 

\begin{equation*}
M(\Phi, Y) = \left[
\begin{array}{cccc}
 \phi_0(y^0) & \phi_1(y^0) & \cdots & \phi_q(y^0)\\
 \phi_0(y^1) & \phi_1(y^1) & \cdots & \phi_q(y^1)\\
 \vdots & \vdots & \vdots & \vdots\\
 \phi_0(y^p) & \phi_1(y^p) & \cdots & \phi_q(y^p)\\
 \end{array}
 \right], 
 %\quad
 \alpha_\Phi = 
\left[
\begin{array}{c}
\alpha_0\\
\alpha_1\\
\vdots\\
\alpha_p\\
\end{array}
\right]
\text{ and }
 f(Y) = 
 \left[
 \begin{array}{c}
 f(y^0)\\
 f(y^1)\\
 \vdots\\
 f(y^p)\\
 \end{array}
 \right].
 \end{equation*}
 We refer to $M(\Phi, Y)$ as the \emph{Vandermonde matrix}. 
If the Vandermonde matrix $M(\Phi, Y)$ is square (i.e., $p=q$, where $q + 1$ is the dimension of $\cP^d_n$) and additionally, $M(\Phi, Y)$ is full rank, then \eqref{eq:get_coeffs} has a unique solution, that is, there exists a unique polynomial $m(x)\in\cP^d_n$ 
given by
$$m(x) = \displaystyle\sum_{j=0}^q \alpha_j^* \phi_j(x)$$
that interpolates $f(x)$ at each point in $Y$;
this unique polynomial is defined by the solution $\alpha^*$ to $M(\Phi, Y)\alpha = f(Y)$. 
We record this in a definition:

\begin{definition}
\label{def:poised}
Given a basis $\Phi$ of dimension $p+1$ for a polynomial space $\cP$ , and a set of points $Y=\{y^0,y^1,\dots,y^p\}$, we say that $Y$ is \emph{poised for interpolation with respect to $\cP$} provided $M(\Phi,Y)$ is invertible. 
\end{definition}

A common and practical choice of basis for $\cP^d_n$ is given by 
$$\Phi_L(y) = \{1, y_1, y_2, \dots, y_n\}$$
when $d=1$ and
$$\Phi_Q(y) = \Phi_L(y) \cup \left\{\frac{y_1^2}{2},\dots,\frac{y_n^2}{2},\frac{y_1 y_2}{\sqrt{2}}, \dots, \frac{y_1 y_n}{\sqrt{2}}, \frac{y_2 y_3}{\sqrt{2}}, \dots, \frac{y_2 y_n}{\sqrt{2}}, \dots, \frac{y_{n-1} y_n}{\sqrt{2}}\right\}$$
when $d=2$. 
In the setting of expensive derivative-free optimization, it is often the case when we employ quadratic models (i.e., $d=2$) that $q = (n+1)(n+2)/2$ function evaluations at points $Y$ -- which would guarantee $M(\Phi_Q, Y)$ is invertible -- are simply not available. 
Moreover, the conditioning of the matrix $M(\Phi, Y)$ must be bounded, as model errors essentially scale linearly with this condition number;
this will be seen, for instance, in \Cref{thm:sfullylinear}.
In model-based DFO, this notion of bounded condition numbers is often formalized through the language of Lagrange polynomials and well-poisedness.
For simplicity in this mansucript, and as we will see in \Cref{thm:sfullylinear}, we will only seek to ensure poisedness with respect to a transformation of the basis $\Phi_L$, which effectively amounts to bounding $\|S\tilde{Y}_{\Delta}\|^{-1}$ where 
$S\in\Reals^{p\times n}$ with $p \leq n$, and where the $i$th column of the $n\times p$ matrix $\tilde{Y}_{\Delta}$ is given by $(y^i - y^0)/\Delta$ for some $\Delta > 0$. 
The fact that $p \leq n$ implies that $S$ is a sketching matrix, hence giving us the term \emph{basis sketching} for the mechanism we introduce in this manuscript. 

Finally, we remind the reader of the definition of \emph{full linearity}, a common (see, e.g., \cite{Conn2009a}[Definition 6.1]) measure of model quality in model-based DFO. 
\begin{definition}
    \label{def:full_linearity}
    Given $\kappa_{ef}, \kappa_{eg} \geq 0$ and $\Delta>0$,
    a model $m(x)$ is a $(\kappa_{ef},\kappa_{eg})$-fully linear model of $f(x)$ on a ball of radius $\Delta$ centered at $x_c$ (i.e., the ball $\cB(x_c; \Delta) := \{y: \|y-x_c\|\leq\Delta\}$) provided
    $$|m(x) - f(x)| \leq \kappa_{ef}\Delta^2 \quad \text{ and } \quad \|\nabla m(x) - \nabla f(x)\| \leq \kappa_{eg}\Delta$$
    for all $y\in \cB(x_c; \Delta)$. 
\end{definition}

Intuitively, \Cref{def:full_linearity} captures the idea that, up to constants, the error of a general model $m(x)$ of $f(x)$ is no worse than the error made by employing a first-order Taylor model of $f(x)$. 

\subsection{Underdetermined Interpolation Systems}\label{sec:underdetermined}
The idea of basis sketching is largely motivated by the idea of minimal norm Hessian (MNH) interpolation introduced in \cite{SW08}.
In the notation established here, MNH assumes the basis $\Phi_Q$ is employed, and partitions the matrix $M(\Phi_Q, Y)$ into two submatrices 
$$M(\Phi_Q,Y) = \left[ M(\Phi_L, Y) \quad M(\Phi_Q \setminus \Phi_L, Y) \right], $$
where ``$\setminus$" denotes the set difference operation. 
The MNH approach replaces \eqref{eq:get_coeffs} with an optimization problem
\begin{equation}
    \label{eq:mnh}
    \displaystyle\min_{\alpha\in\Reals^{n+1}, \beta\in\Reals^{n(n+1)/2}}\left\{\frac{1}{2}\|\beta\|^2 : M(\Phi_L, Y)\alpha +  
    M(\Phi_Q \setminus \Phi_L, Y)\beta = f(Y) \right\}.
\end{equation}
In words, \eqref{eq:mnh} removes the degrees of freedom from the underdetermined linear system in the constraint of \eqref{eq:mnh} by choosing the basis coefficients corresponding to degree two polynomials in the basis such that these coefficients are of minimal norm. 
The work in \cite{SW08} uses the KKT conditions of \eqref{eq:mnh} to develop a method based on an iteratively updated QR factorization of $M(\Phi_L, Y)$ to choose a set $Y$ such that $M(\Phi_L, Y)$ is well-conditioned. 
The selection of $Y$ in the QR-based method of \cite{SW08} greedily selects points from a bank of points for which $f$ has previously been evaluated, and then adds points to $Y$ (determined from the nullspace suggested by the QR factorization) for which function evaluations must be performed only after the greedy procedure has been run. 

Basis sketching will mimic this greedy QR-based procedure of \cite{SW08}, but will first modify the subproblem \eqref{eq:mnh} with a sketch.
In particular, given a sketching matrix $S\in\Reals^{p\times n}$ with mutually orthonormal rows $s_1,\dots, s_p$, we define a basis 
$$\Phi_S = \left\{1, s_1 y, \dots, s_p y\right\}$$
for a subspace of $\cP^1_n$. 
Because we assume $S$ has mutually orthornormal rows, we can choose a perpendicular matrix $S^\perp\in\Reals^{(n-p)\times n}$ satisfying
$s_j^\perp s_i^\top = 0$ for $i=1,\dots, p$ and $j=1,\dots, n-p$, where $j$ indexes the rows of $S^\perp$. 
This gives us a basis for an orthogonal subspace of $\cP^1_n$,
$$\Phi_{S^\perp} = \left\{s^\perp_1 y, \dots, s^\perp_{n-p} y \right\}.$$

With this notation, our subproblem of interest is 
\begin{equation}
    \label{eq:bask}
    \displaystyle\min_{\alpha\in\Reals^{p + 1}, \beta\in\Reals^{n(n+2)/2},  \gamma\in\Reals^{n-p}}\left\{ \frac{1}{2}\|\beta\|^2 + \frac{1}{2}\|\gamma\|^2 : 
    M(\Phi_S, Y)\alpha + M(\Phi_{S^\perp}, Y)\gamma + M(\Phi_Q \setminus \Phi_L, Y)\beta = f(Y) \right\}.
\end{equation}
In words, \eqref{eq:bask} fits an underdetermined quadratic model by limiting the degrees of freedom not only by minimizing the contribution from the degree two polynomials in the basis, but also from the subspace of $\cP^1_n$ spanned by $\Phi_{S^\perp}$. 

To motivate future development in this manuscript, we additionally mention that the subproblem \eqref{eq:mnh} is trivially extended to a problem of identifying a minimal \emph{change} Hessian, that is, replacing the objective with $\frac{1}{2}\|\beta - \bar\beta\|^2$, where $\bar\beta$ denotes the coefficients on $\Phi_Q \setminus \Phi_L$ taken from, for instance, the model Hessian employed in a previous iteration of an optimization algorithm. 
In fact, this trivial extension of \eqref{eq:mnh} is what is employed in \texttt{POUNDers} \cite{SWCHAP14} for the purpose of model-building. 
Any method for solving \eqref{eq:mnh} can be extended to solving the minimal change problem simply by replacing each entry of the right hand side of \eqref{eq:mnh} with a residual term
$$f(y^i) - \displaystyle\sum_{j=1}^{\frac{n(n+1)}{2}} \bar\beta_j \phi_j(y^i),$$
where the $n(n+1)/2$ basis functions $\phi_j$ are from $\Phi_Q \setminus \Phi_L$. 
Similarly, if we replace the objective of \eqref{eq:bask} with 
$\frac{1}{2}\|\beta - \bar\beta\|^2 + \frac{1}{2}\|\gamma - \bar\gamma\|^2$,
then any method for solving \eqref{eq:bask} can be extended to this modified subproblem by replacing each entry in the right hand side of \eqref{eq:bask} with the appropriate residual term.
%The specifics of this residual modification to \eqref{eq:bask} will be discussed later. 

Our proposed basis sketching method is essentially a \emph{sketch-and-project} process \cite{gower2015randomized}.
Basis sketching will maintain an \emph{average estimate} of the gradient 
$\bar{g}\in\Reals^n$. 
On a given iteration of our basis sketching method, we will, in general, not construct a fully linear model of $f$, 
but will instead construct a model that is only fully linear when restricted to a particular subspace of $\Reals^n$ defined by a matrix $S$. 
We formalize this with an extension of the full linearity definition \Cref{def:full_linearity} which we call $S$-full linearity. 
While our definition of $S$-full linearity is similar to the definition of $Q$-full linearity suggested in recent works in subspace methods for DFO \cite{CRsubspace2021, STARS2022}, note that they are not quite the same.
\begin{definition}
    \label{def:s_full_linearity}
    Given $S\in\Reals^{p\times n}$ with $p\leq n$, constants
    $\kappa_{ef}, \kappa_{eg} \geq 0$ and $\Delta>0$,
    a model $m(x)$ is a $(S,\kappa_{ef},\kappa_{eg})$-fully linear model of $f(x)$ on $\cB(x_c; \Delta)$ provided
    $$|m(x_c + S^\top d) - f(x_c + S^\top d)| \leq \kappa_{ef}\Delta^2 \quad \text{ and } \quad \|\nabla m(x_c + S^\top d) - \nabla f(x_c + S^\top d)\| \leq \kappa_{eg}\Delta$$
    for all 
    $d\in\Reals^p$ satisfying $\|S^\top d\|\leq\Delta$. 
    %$d\in \cB(0_p; \frac{1}{\sqrt{np}}\Delta)$. 
\end{definition}

As is often seen in the literature with full linearity, we will often drop the constants $\kappa_{ef}, \kappa_{eg}$ when discussing $S$-full linearity. 
We now prove an extension of a result in \cite{WildPhD, SWRRCS07}
which provides sufficient conditions to guarantee $S$-full linearity given some general conditions on the interpolation set $Y$. 
We first state an assumption made throughout this paper. 
\begin{assumption}
    \label{ass:cd} 
    The function $f: \Reals^n\to \Reals$ is Lipschitz continuously differentiable on $\cD$.
    In particular, there exist $0 \leq L_f, L_g < \infty$ such that
    \begin{itemize}
    \item $|f(x) - f(y)| \leq L_f\|x-y\|$ and 
    \item $\|\nabla f(x) - \nabla f(y)\| \leq L_g\|x-y\|$
    \end{itemize}
    for all $x,y\in \cD$. 
\end{assumption}
\begin{theorem}
\label{thm:sfullylinear}
Let $\Delta > 0$, let $\Lambda > 0 $, and let $S\in\Reals^{p\times n}$ with $p\leq n$ have mutually orthonormal rows.
Let $\{y^0,y^1,\dots,y^p\} \subset \cB(y^0; c\Delta)$ for some $c>0$ and suppose that there exists $\delta^i\in\Reals^p$ so that $S\delta^i = y^i-y^0$ for each $i=1,\dots,p$. 
Let $\tilde{Y}$ denote an $p\times p$ matrix where the $i$th column is given by $\delta^i$. 
Suppose $f$ satisfies \Cref{ass:cd} and additionally suppose $m$ is twice continuously differentiable on $\cB(y^0;c\Delta)$ with gradient Lipschitz constant $L_{mg}$\footnote{In the case of quadratic models considered in this paper, $L_{mg}$ is trivially derived from the spectral norm of the model Hessian.}.
If both
\begin{enumerate}
\item $\|\tilde{Y}\|^{-1} \leq \displaystyle\frac{\Lambda}{c\Delta}$ and
\item $f(y^i) = m(y^i)$ for $i=0,1,\dots,p$, 
\end{enumerate}
then $m(x)$ is an $(S, \kappa_{ef}, \kappa_{eg})$-fully linear model of $f(x)$ with constants
$$\kappa_{ef} = \displaystyle\frac{4 + 5\Lambda\sqrt{p}}{2}(L_g + L_{mg})c^2 \quad \text{ and} \quad \kappa_{eg} = \frac{5\Lambda\sqrt{p}}{2}(L_g + L_{mg})c.$$
\end{theorem}

We defer the proof to \Cref{sec:proof1} for readability. 

\section{Average and Ameliorated Estimators}\label{sec:estimators}
We now present the estimators that will be used in the basis sketching algorithm. 
\Cref{sec:sketch_and_project} will discuss the motivation behind \emph{average estimators}, a generally biased estimator of a gradient $\nabla f(x)$.
In turn, \Cref{sec:ameliorated} will then discuss a modification we will employ to yield \emph{ameliorated estimators}, which are, by construction, an unbiased estimator for the gradient. 
Finally, \Cref{sec:minimum_variance} will discuss how (ideally) one would yield a minimum variance unbiased estimator for the gradient. 
\subsection{Average Estimators: Sketch and Project}\label{sec:sketch_and_project}
The basis sketching trust region method is an iterative method (iterations indexed by $k$) that maintains an \emph{average estimate} $\bar{g}^k\in\Reals^n$ of the gradient $\nabla f(x^k)$, where $x^k$ is an incumbent point held by the algorithm. 
On each iteration $k$, we will identify 
(via a greedy procedure similar to the one employed in the MNH method) 
an orthonormal matrix $Q_k\in\Reals^{n\times n}$.
We will then (randomly) select a 
subset $J_k\subset\{1,2,\dots,n\}$ of size $p_k$.
Note that the value of $p_k\leq n$ may change on each iteration.
We then choose the sketching matrix $S_k = ([Q_k]_{J_k})^\top$, that is, $S_k$ is the transpose of the submatrix consisting of the columns of $Q_k$ indexed by $J_k$.
Then, motivated by \Cref{thm:sfullylinear} and using the notation of that theorem, we select a set $Y_k$ of $p_k$ many points so that $\|\tilde{Y_k}\|^{-1}$ is sufficiently bounded from above and so that each point in $Y_k$ is sufficiently close to $x^k$. 
We then 
perform any necessary function evaluations at the points in $Y_k$, noting that the selection of $Q_k$ is designed to encourage that the function evaluations being performed are relatively few.
With this data, we obtain a model from the solution to \eqref{eq:bask} where the two bases for the linear polynomials are defined as $\Phi_{S_k}$ and $\Phi_{S_k^\perp}$; we note that by the construction of the orthonormal $Q_k$, obtaining the matrix $S_k^\perp$ is immediate. 
Because the constraint of \eqref{eq:bask} enforces that $m(y^i) = f(y^i)$ for all $y^i\in Y_k$, 
we have from \Cref{thm:sfullylinear} that the quadratic model $m(x)$ derived from the solution to \eqref{eq:bask} is a $S_k$-fully linear model of $f(x)$ on the ball $\cB(x^k,\Delta_k)$. 
Let $\hat{g}^k\in\Reals^n$ denote the gradient term of the model $m(x)$. 
We then consider a subproblem to \emph{update} the average gradient,
\begin{equation}
\label{eq:sketch_and_project}
\bar{g}^k := 
\begin{array}{rl}
\displaystyle\arg\min_{\alpha\in\Reals^n} &\frac{1}{2}\|\alpha - \bar{g}^{k-1}\|^2\\
 \text{s. to} & S_k \alpha = S_k\hat{g}^k\\
\end{array}
% \quad \text{ and } \quad
% \bar{H}^k := \mat\left(
% \begin{array}{rl}
% \displaystyle\arg\min_{\alpha\in\Reals^{n\times n}} &\frac{1}{2}\|\vec(\alpha) - \vec(\bar{H}^{k-1})\|^2\\
%  \text{s. to} & S_k^\top \alpha S_k = \hat{H}^{I_k}\\
% \end{array}
% \right)
\end{equation}

In words, \eqref{eq:sketch_and_project} selects $\bar{g}^k$ as the closest (in Euclidean norm) vector to $\bar{g}^{k-1}$ such that the $S_k$-sketch of $\bar{g}^k$ agrees with the $S_k$-sketch of the ($S_k$-fully linear) model gradient $\hat{g}^k$. 
%Likewise, \eqref{eq:sketch_and_project} selects $\bar{H}^k$ as the closest (in Frobenius norm) matrix to $\bar{H}^{k-1}$ such that the two-sided sketch of $\bar{H}^k$ agrees with the model Hessian $\hat{H}^{I_k}$. 
We quickly prove that \eqref{eq:sketch_and_project} has a closed-form solution. 

\begin{proposition}
\label{prop:closed-form}
A closed form solution to \eqref{eq:sketch_and_project} exists and is given as
 \begin{equation}
\label{eq:closed-form-g}
\bar{g}^k = \bar{g}^{k-1} - S_k^\top S_k \bar{g}^{k-1} + S_k^\top S_k\hat{g}.
\end{equation}
% Also, a closed form solution exists for $\bar{H}^k$ and is given as
% \begin{equation}
% \label{eq:closed-form-H}
% \bar{H}^k = \bar{H}^{k-1} - S_k S_k^\top \bar{H}^{k-1} S_k S_k^\top + S_k\hat{H}^{I_k} S_k^\top.
% \end{equation}
\end{proposition}

\begin{proof}
The KKT conditions associated with \eqref{eq:sketch_and_project} can be expressed as
\begin{equation*}
\begin{array}{rll}
\alpha & = \bar{g}_{k-1} - S_k^\top \mu & (stationarity)\\
S_k\alpha& = S_k\hat{g}^k & (primal feasibility),\\
\end{array}
\end{equation*}
where $\mu\in\Reals^{p_k}$ is a vector of Lagrange multipliers.
Plugging the stationarity condition into the primal feasibility condition,
$S_k \bar{g}^{k-1} - S_k S_k^\top \mu = S_k\hat{g}^{k}$.
Solving for $\mu$ and using the orthonormality of $Q_k$, we obtain
$\mu = S_k(\bar{g}^{k-1} -  \hat{g}^k).$
Plugging these Lagrange multipliers $\mu$ back into the stationarity condition, we obtain \eqref{eq:closed-form-g}. 
Because the constraints of the problem determining $\bar{g}^k$ in \eqref{eq:sketch_and_project} are affine, and because the objective is convex in $\alpha$, the KKT conditions are also sufficient for optimality. 

% The proof of the second claim is essentially the same.
% The KKT conditions for determining $\bar{H}^k$ in \eqref{eq:sketch_and_project} can be expressed as
% \begin{equation*}
% \begin{array}{rll}
% \vec(\alpha) & = \vec(\bar{H}^{k-1}) - (S_k\otimes S_k)\mu & (stationarity)\\
% (S_k^\top \otimes S_k^\top)\vec(\alpha) & = \vec(\hat{H}^{I_k}) & (primal feasibility),\\
% \end{array}
% \end{equation*}
% Plugging the stationary condition into the primal feasibility condition,
% $$(S_k^\top \otimes S_k^\top)\vec(\bar{H}^{k-1}) - (S_k^\top \otimes S_k^\top)(S_k \otimes S_k)\mu = \vec(\hat{H}^{I_k}).$$
% Using properties of the Kronecker product and the orthonormality of $Q_k$, 
% this simplifies to
% $$\mu = (S_k^\top \otimes S_k^\top)\vec(\bar{H}^{k-1})- \vec(\hat{H}^{I_k}).$$
% Plugging these Lagrange multipliers back into the stationarity condition, we obtain
% \begin{equation*}
% \begin{array}{ll}
% \vec(\alpha) & = \vec(\bar{H}^{k-1}) - (S_k\otimes S_k) (S_k^\top \otimes S_k^\top)\vec(\bar{H}^{k-1})- (S_k\otimes S_k) \vec(\hat{H}^{I_k})\\
% & = \vec(\bar{H}^{k-1}) - (S_kS_k^\top \otimes S_kS_k^\top) \vec(\bar{H}^{k-1})- (S_k\otimes S_k) \vec(\hat{H}^{I_k})\\
% & = \vec(\bar{H}^{k-1}) - S_k S_k^\top \bar{H}^{k-1} S_k S_k^\top - S_k \hat{H}^{I_k} S_k^\top,
% \end{array}
% \end{equation*}
% where we have used the property that for appropriately sized matrices $A, B,$ and $C$, $ABC = \mat((C^\top\otimes A)\vec(B))$.
\end{proof}

%We pause to remark on what \Cref{prop:closed-form} and \Cref{prop:changebasis} imply for a method that employs polynomial model updates based on \eqref{eq:closed-form}. 
%In each iteration of a basis sketching method, the update of a polynomial via \eqref{eq:closed-form} only entails $s_k$ many function evaluations. 
%Moreover, there is conceptually \emph{never} a reason to evaluate all the points of $Y_k$ - one can recompute $Y_k$ on each iteration to be -- for example -- a perfectly 1-poised set of points that includes a current point $x^k$ as $y^0$.
%We additionally remark that by fixing the natural basis $\bar\Phi$ as a reference basis independent of $k$, the computation of model derivatives at $y^0$, which would allow us to define a Taylor-like approximation to $f$ centered at $y^0$, is straightforward. 
%Finally, while transforming coefficients $\bar\alpha^k_{\lambda}$ from \eqref{eq:closed-form} into coefficients in the natural basis requires inverting $M(\bar\Phi,Y)$ according to \Cref{prop:changebasis},we are implicitly assuming throughout this paper that the computational cost of individual evaluations of $f$ far exceeds the computational cost of the inversion of a matrix in $\Reals^{p_1\times p_1}$. 

\subsection{Ameliorated Estimators} \label{sec:ameliorated}
At this point in our development, one could imagine and fully describe an iterative randomized method that essentially uses the model determined by the subproblem \eqref{eq:bask} in each iteration, but replaces the model gradient $\hat{g}^k$ with the average estimate 
$\bar{g}^k$ updated by \eqref{eq:closed-form-g}.
However, in the same sense that SAG \cite{Schmidt2013} estimators are biased estimators of the true gradient, $\nabla f(x^k)$, $\bar{g}^k$ 
%(resp. $\bar{H}^k)$ 
is a biased estimator of $\nabla f(x^k)$.% (resp. $\nabla^2 f(x^k)$).  
 We now describe how to slightly modify 
 the update \eqref{eq:closed-form-g} 
 %and \eqref{eq:closed-form-H}, 
 via control variates so as to attain an unbiased estimator.  

In the $k$th iteration of the basis sketching method, we associate with each $i\in\{1,\dots,n\}$ an independent Bernoulli variable with success probability $\pi_i > 0$. 
Each of the $n$ Bernoulli trials are realized, and for each successful trial, we include $i\in J_k$. 
%For notational ease, denote 
%$$g^{I_k}_\sharp := [I_n]_{I_k}\hat{g}^{I_k} \quad \text{ and } \quad H^{I_k}_\sharp = [I_n]_{I_k}\hat{H}^{I_k}[I_n]^\top_{I_k}.$$ 
%With this sharp notation, the updates in \eqref{eq:closed-form-g} and \eqref{eq:closed-form-H} can be equivalently written as
Letting $q_k^i$ denote the $i$th column of $Q_k$, 
we note that the update in \eqref{eq:closed-form-g} can be equivalently written 
$$\bar{g}^k = \bar{g}^{k-1} - 
\displaystyle\sum_{i\in J_k} \left[q_k^i q_k^{i\top}\right]\bar{g}^k + 
\displaystyle\sum_{i\in J_k} \left[q_k^i q_k^{i\top}\right]\hat{g}^k.$$
%and
%$$\bar{H}^k = \bar{H}^{k-1} - 
%\displaystyle\sum_{i\in I_k} \left[[Q_k]_i[Q_k]_i^\top\right]\bar{H}^{k-1}\left[[Q_k]_i[Q_k]_i^\top\right] 
%+ \displaystyle\sum_{(i,j)\in I_k\times I_k} [H^{I_k}_\sharp]_{i,j}[Q_k]_i[Q_k]^\top_j.
%$$
Thus, one can see how
\begin{equation}
\label{eq:tildegk}
\tilde{g}^k := \bar{g}^{k-1} -
\displaystyle\sum_{i\in J_k} \frac{1}{\pi_i}\left[q_k^i q_k^{i\top}\right]\bar{g}^{k-1} + 
\displaystyle\sum_{i\in J_k} \frac{1}{\pi_i}\left[q_k^i q_k^{i\top}\right]\hat{g}^k.
\end{equation}
% and 
% \begin{equation}
% \label{eq:tildeHk}
% \tilde{H}^k := \bar{H}^{k-1} 
% - \displaystyle\sum_{i\in I_k} \frac{1}{\pi_i^2}\left[[Q_k]_i[Q_k]_i^\top\right]\bar{H}^{k-1}\left[[Q_k]_i[Q_k]_i^\top\right] 
% + \displaystyle\sum_{(i,j)\in I_k\times I_k} \frac{1}{\pi_i \pi_j}[H^{I_k}_\sharp]_{i,j}[Q_k]_i[Q_k]^\top_j.
% \end{equation}
is a particular reweighting of the closed-form update \eqref{eq:closed-form-g}. %and \eqref{eq:closed-form-H}, respectively. 
To simplify notation, given a vector of nonzero probabilities $\pi^k\in\Reals^{n}$ and the realization $J_k$, we define a diagonal matrix $D(\pi^k, J_k)\in\Reals^{p_k\times p_k}$ (recall that $|J_k| = p_k$) via 
$$D(\pi^k,J_k) = diag([1/\pi_i^k: i\in J_k]).$$
We can then rewrite \eqref{eq:tildegk} as %and \eqref{eq:tildeHk} as
\begin{equation}
\label{eq:tildegk2}
\tilde{g}^k =  \bar{g}^{k-1} - S_k^\top D(\pi^k, J_k) S_k \bar{g}^{k-1} + S_k^\top D(\pi^k, J_k) S_k \hat{g}^k.
\end{equation}
% and
% \begin{equation}
% \label{eq:tildeHk2}
% \tilde{H}^k =  \bar{H}^{k-1} - S_k D(\pi^k,I_k) S_k^\top \bar{H}^{k-1} S_k D(\pi^k,I_k) S_k^\top 
% + S_k D(\pi^k,I_k)\hat{H}^{I_k} D(\pi^k, I_k) S_k^\top. 
% \end{equation}
We remark that there are no further restrictions on $\pi^k$ other than all entries be nonzero, and the following theorems in this subsection hold under no further assumptions. 
The structure of the update \eqref{eq:tildegk2} is inspired by \emph{arbitrary sampling}; for papers on the subject of arbitrary sampling within the related context of (block) coordinate descent, see \cite{allen2016even, qu2016coordinate, RichtarikTakac2016, HanzelyRichtarik2019}. 

We now record the following important observation, namely, that $\tilde{g}^k$ is an \emph{unbiased estimator} of a $\mathcal{O}(\Delta_k)$-accurate approximation of $\nabla f(x^k)$. 

\begin{theorem}
\label{thm:unbiased}
Let $\kappa_{ef},\kappa_{eg} \geq 0$.
Suppose, for each $J_k$, that we obtain a model $m_{J_k}$ such that $m_{J_k}$ is a $(S_{J_k},\kappa_{ef},\kappa_{eg})$ fully linear model of $f$ on $\cB(x^k,\Delta_k)$. 
Let $\hat{g}_{J_k}$ denote the model gradient $\nabla m_{J_k}(x^k)$. 
Then, the expectation of $\tilde{g}^k$ with respect to the probability distribution on the selection of $J_k$ satisfies
$$\|\mathbb{E}_{J_k}[\tilde{g}^k] - \nabla f(x)\| \leq \sqrt{n}\kappa_{eg}\Delta_k$$
for all $x\in \cB(x^k,\Delta_k)$. 
\end{theorem}

\begin{proof}
Starting from the definition \eqref{eq:tildegk},
\begin{equation*}
\begin{array}{rl}
\mathbb{E}_{J_k}\left[\tilde{g}^k\right] = & \bar{g}^{k-1} - 
\mathbb{E}_{J_k}\left[\displaystyle\sum_{i\in J_k} \frac{1}{\pi_i}\left[q_k^i q_k^{i\top}\right]\bar{g}^{k-1}\right] 
+ \mathbb{E}_{J_k}\left[\displaystyle\sum_{i\in J_k} \frac{1}{\pi_i}\left[q_k^i q_k^{i\top}\right]\hat{g}_{J_k}\right]\\
= & \bar{g}^{k-1} - 
\mathbb{E}_{J_k}\left[\displaystyle\sum_{i=1}^{n} \one\left[i\in J_k \right]\frac{1}{\pi_i}\left[q_k^i q_k^{i\top}\right]\right]\bar{g}^{k-1}
+ \mathbb{E}_{J_k}\left[\displaystyle\sum_{i\in J_k} \frac{1}{\pi_i}\left[q_k^i q_k^{i\top}\right]\hat{g}_{J_k} \right]\\
= & \bar{g}^{k-1} - 
\displaystyle\sum_{i=1}^n \left[\pi_i \frac{1}{\pi_i} \left[q_k^i q_k^{i\top}\right]\right]\bar{g}^{k-1}
+ \mathbb{E}_{J_k}\left[\displaystyle\sum_{i\in J_k} \frac{1}{\pi_i}\left[q_k^i q_k^{i\top}\right]\hat{g}_{J_k} \right]\\
= & \bar{g}^{k-1} - \bar{g}^{k-1} + \mathbb{E}_{J_k}\left[\displaystyle\sum_{i\in J_k} \frac{1}{\pi_i}\left[q_k^i q_k^{i\top}\right]\hat{g}_{J_k} \right]\\\\
= & \mathbb{E}_{J_k}\left[\displaystyle\sum_{i\in J_k} \frac{1}{\pi_i}\left[q_k^i q_k^{i\top}\right]\hat{g}_{J_k} \right]\\
= & \displaystyle\sum_{J_k} p(J_k) \sum_{i\in J_k} \frac{1}{\pi_i}\left[q_k^i q_k^{i\top}\right]\hat{g}_{J_k}.\\
\end{array}
\end{equation*}

Now consider $Q_k^\top\left(\mathbb{E}_{J_k}[\tilde{g}^k] - \nabla f(x)\right)$
for any $x\in \cB(x^k, \Delta_k)$. 
By orthonormality, the $j$th coordinate of $Q_k^\top\mathbb{E}_{J_k}[\tilde{g}^k]$ is 
$$\left|Q_k^\top [\mathbb{E}_{J_k}[\tilde{g}^k]]_j\right| = 
\displaystyle\sum_{J_k: j\in J_k} p(J_k) \frac{1}{\pi_j} q_k^{j\top} \hat{g}_{J_k}.$$
Thus, 
$$\left|Q_k^\top [\mathbb{E}_{J_k}[\tilde{g}^k] - \nabla f(x)]]_j\right| = 
q_k^{j\top}\left(\left[\frac{1}{\pi_j} \displaystyle\sum_{J_k: j\in J_k} p(J_k)  \hat{g}_{J_k}\right] - \nabla f(x)\right)$$
Because $\displaystyle\sum_{J_k:j\in J_k}p(J_k) = \pi_j$, we can equivalently write
$$\left|Q_k^\top [\mathbb{E}_{J_k}[\tilde{g}^k] - \nabla f(x)]]_j\right| = 
q_k^{j\top}\left[\frac{1}{\pi_j} \displaystyle\sum_{J_k: j\in J_k} p(J_k) (\hat{g}_{J_k}- \nabla f(x))\right] $$
By our supposition that each $\hat{g}_{J_k}$ is the gradient of a $S_{J_k}$-fully-linear model,
$$
\left|Q_k^\top [\mathbb{E}_{J_k}[\tilde{g}^k] - \nabla f(x)]]_j\right|
 \leq \|q_k^j\| \frac{1}{\pi_j}
\left|\displaystyle\sum_{J_k: j\in J_k} p(J_k) (\hat{g}_{J_k}- \nabla f(x)) \right|
\leq \frac{1}{\pi_j} \displaystyle\sum_{J_k: j\in J_k} p(J_k)\|\hat{g}_{J_k}- \nabla f(x^k)\| \leq \kappa_{eg}\Delta_k. 
 $$
Thus, 
$$\|\mathbb{E}_{J_k}[\tilde{g}^k] - \nabla f(x)\| = \|Q_k^\top\left(\mathbb{E}_{J_k}[\tilde{g}^k] - \nabla f(x)\right)\| \leq \sqrt{n}\kappa_{eg}\Delta_k,$$
as we meant to show. 
\end{proof}

\subsection{Towards A Minimum Variance Estimator} \label{sec:minimum_variance}
Having established the sense in which $\tilde{g}^k$ is an unbiased estimator of a particular approximation of $\nabla f(x)$, we now state a result concerning the variance of this estimator.
\begin{theorem}
\label{thm:variance}
The variance of $\tilde{g}^k$, with respect to the distribution governing $J_k$, is
\begin{equation}
\label{eq:variance}
\mathbb{E}_{J_k}\left[\|\tilde{g}^k - \Egk\|^2\right] = \|\bar{g}^{k-1} - \Egk\|^2_{Q_k D_k Q_k^\top - I_{n}},\end{equation}
where we denote $D_k=D(\pi^k,\{1,\dots,n\})$ and $I_n$ denotes the $n$-dimensional identity matrix.
\end{theorem} 

\begin{proof}
Denote 
$$P(J_k) = \displaystyle\sum_{i=1}^n\one\left[i\in J_k\right]\frac{1}{\pi_i}q_i q_i^\top \quad \text{ and } \quad
v(J_k) = P(J_k)\hat{g}_{J_k}$$
We first record that
\begin{equation}\label{eq:var_p1}
\begin{array}{rcl}
\mathbb{E}_{J_k}\left[ \|\tilde{g}^k -\Egk\|^2\right] & =  & \mathbb{E}_{J_k} \left[\left\|[I-P(J_k)]\bar{g}^{k-1} + v(J_k)\right\|^2\right] - \|\Egk\|^2\\
& = & \mathbb{E}_{J_k}\left[(\bar{g}^{k-1})^\top P(J_k)^\top P(J_k)\bar{g}^{k-1} - 2(\bar{g}^{k-1})^\top P(J_k)\bar{g}^{k-1}\right. \\
&& \left.+2(\bar{g}^{k-1})^\top v(J_k) - 2(\bar{g}^{k-1})^\top P(J_k) v(J_k) + \|v(J_k)\|^2\right] +  \|\bar{g}^{k-1}\|^2 - \|\Egk\|^2 \\
\end{array}
\end{equation}
We now compute the expectations, with respect to $J_k$, of the matrices $P(J_k)$ and $P(J_k)^\top P(J_k)$, of the vectors $P(J_k)v(J_k)$ and $v(J_k)$, and of the scalar $\|v(J_k)\|^2$.

\begin{equation*}
\mathbb{E}_{J_k}\left[ P(J_k)\right] = 
\mathbb{E}_{J_k}\left[ \displaystyle\sum_{i=1}^n \one[i\in J_k] \frac{1}{\pi_i^k} q_i q_i^\top \right] = 
\displaystyle\sum_{i=1}^n q_i q_i^\top =  Q_kQ_k^\top = I_n
\end{equation*}

\begin{equation*}
\begin{array}{rcl}
\mathbb{E}_{J_k}\left[ P(J_k)^\top P(J_k)\right] & = & 
\mathbb{E}_{J_k}\left[ \left[\displaystyle\sum_{i=1}^n \one[i\in J_k] \frac{1}{\pi_i^k} q_i q_i^\top \right]^\top
\left[ \displaystyle\sum_{j=1}^n \one[j\in J_k] \frac{1}{\pi_j^k} q_j q_j^\top \right]\right] \\

& = & \mathbb{E}_{J_k} \left[\displaystyle\sum_{i=1}^n \displaystyle\sum_{j=1}^n \one[i,j\in J_k] \frac{1}{\pi_i^k\pi_j^k} \left[q_i q_i^\top \right] q_j q_j^\top \right] \\

& = & \mathbb{E}_{J_k} \left[\displaystyle\sum_{i=1}^n \one[i\in J_k]\frac{1}{(\pi^k_i)^2} q_i q_i^\top \right]\\

& = & \displaystyle\sum _{i=1}^n \frac{1}{\pi^k_i} q_i q_i^\top = Q_kD_kQ_k^\top \\ %= D_k\\
\end{array}
\end{equation*}

\begin{equation*}
\begin{array}{rcl}
\mathbb{E}_{J_k}\left[P(J_k)v(J_k)\right] & = &
\mathbb{E}_{J_k}\left[\left[\displaystyle\sum_{i=1}^n\one\left[i\in J_k\right]\frac{1}{\pi^k_i}q_iq_i^\top\right]
\left[\displaystyle\sum_{i=1}^n\one\left[i\in J_k\right]\frac{1}{\pi^k_i}q_iq_i^\top \hat{g}^{J_k}\right]
\right]\\

& = & \mathbb{E}_{J_k}\left[ \displaystyle\sum_{i=1}^n \displaystyle\sum_{j=1}^n \one\left[i,j\in J_k\right] \frac{1}{\pi^k_i\pi^k_j} q_iq_i^\top q_jq_j^\top \hat{g}^{J_k}\right]\\

& = & \mathbb{E}_{J_k}\left[ \displaystyle\sum_{i=1}^n \one\left[i\in J_k\right] \frac{1}{(\pi^k_i)^2} q_i q_i^\top \hat{g}^{J_k}\right]\\

& = & Q_k D_k Q_k^\top \Egk
\end{array}
\end{equation*}

\begin{equation*}
\mathbb{E}_{J_k}\left[ v(J_k) \right] = 
\mathbb{E}_{J_k} \left[\displaystyle\sum_{i=1}^n\one\left[i\in J_k\right]\frac{q_i^\top \hat{g}^{J_k}}{\pi^k_i}q_i\right] =
\Egk
\end{equation*}

\begin{equation*}
\mathbb{E}_{J_k}\left[ \|v(J_k)\|^2 \right] = 
\mathbb{E}_{J_k} \left[\displaystyle\sum_{i\in J_k} \frac{(q_i^\top \hat{g}^{J_k})^2}{(\pi^k_i)^2}q_i^\top q_i\right] = 
\displaystyle\sum_{J_k} p(J_k)  \frac{(q_i^\top \hat{g}^{J_k})^2}{(\pi^k_i)^2}
= \Egk^\top Q_k D_k Q_k^\top \Egk
\end{equation*}
We can now continue the equalities in \eqref{eq:var_p1}:

\begin{equation}
\begin{array}{rcl}
\mathbb{E}_{J_k}\left[\|\tilde{g}^k-\Egk\|^2 \right] & = &
(\bar{g}^{k-1})^\top Q_k D_k Q_k^\top \bar{g}^{k-1}
- 2(\bar{g}^{k-1})^\top \bar{g}^{k-1}
+2(\bar{g}^{k-1})^\top\Egk \\
& & -2(\bar{g}^{k-1})^\top D_k\Egk + \Egk^\top Q_k D_k Q_k^\top \Egk + \|\bar{g}^{k-1}\|^2 - \|\Egk\|^2\\
& = & \|\bar{g}^{k-1} - \Egk\|^2_{Q_k D_k Q_k^\top -I_n},\\
\end{array}
\end{equation}
as we intended to show. 
\end{proof}

We make three observations concerning \Cref{thm:variance}. 
First, as a sanity check, notice that because $D_k\succ I_n$ by the definition of $D_k$, we have that $Q_k D_k Q_k^\top \succ I_n$, and so the quantity \eqref{eq:variance} is always nonnegative.
Second, if $D_k = I_n$, then the variance is zero; this makes sense because this means each $\pi^k_i=1$, i.e., deterministically, $J_k = \{1,2,\dots,p\}$, and so the variance is trivial. 
Third, and of practical importance, the variance of a (nontrivial) estimator will be generally unknown, since the quantity $\mathbb{E}_{J_k}[\tilde{g}^k]$ (and its expectation, a $\mathcal{O}(\Delta_k)$-accurate approximation of $\nabla f(x^k)$) appearing in the right hand side of \eqref{eq:variance} is unknown. 
This is of course a practical concern, and we will provide a proxy for this unknown quantity in \Cref{sec:practical}, but we will first discuss how to derive an estimator $\tilde{g}^k$ of minimum variance, assuming this quantity were known. 

Observe first that, by the assumed independence of the Bernoulli variables, 
$$\mathbb{E}\left[|J_k|\right] = \displaystyle\sum_{i=1}^{n} \pi^k_i.$$
We consider an optimization problem to minimize the variance of $\tilde{g}^k$ under the constraint that the \emph{expected} size of $J_k$ is a given value, $p_k \leq n$. 
That is, assuming we had access to $\delta^k := \bar{g}^{k-1} - \mathbb{E}_{J_k}[\tilde{g}^k]$,
we would 
solve

\begin{equation}\label{eq:opt_prob}
\begin{array}{rl}
\displaystyle\min_{\pi^k} & \|\delta^k\|^2_{Q_k D_k Q_k^\top - I_n}\\
\text{s. to} & \displaystyle\sum_{i=1}^n {\pi^k_i} = p_k\\
& 0 \leq \pi^k_i \leq 1 \quad \forall i
\end{array}
\equiv
\begin{array}{rl}
\displaystyle\min_{\pi^k} & \|Q_k^\top \delta^k\|^2_{D_k - I_n}\\
\text{s. to} & \displaystyle\sum_{i=1}^n {\pi^k_i} = p_k\\
& 0 \leq \pi^k_i \leq 1 \quad \forall i
\end{array}
\begin{array}{rl}
\displaystyle\min_{\pi^k} & \displaystyle\sum_{i=1}^{n} \left(\frac{1}{\pi^k_i} - 1\right)\left[Q_k^\top\delta^k\right]_i^2 \\
\text{s. to} & \displaystyle\sum_{i=1}^n {\pi^k_i} = p_k\\
& 0 \leq \pi^k_i \leq 1 \quad \forall i
\end{array}
\equiv
\begin{array}{rl}
\displaystyle\min_{\pi^k} & \displaystyle\sum_{i=1}^n \frac{\left[Q_k^\top\delta^k\right]^2_i}{\pi^k_i} \\
\text{s. to} & \displaystyle\sum_{i=1}^n {\pi^k_i} = p_k\\
& 0 \leq \pi^k_i \leq 1 \quad \forall i
\end{array}
\end{equation}
The following theorem is immediate by deriving KKT conditions. 

\begin{theorem}
\label{thm:min_variance}
The optimal solution of \eqref{eq:opt_prob} is expressible in closed form and is given, for each $i$, as
\begin{equation}\label{eq:prob_dist} 
\pi^k_{(i)} = \left\{
\begin{array}{rl}
(p_k+c - n)\displaystyle\frac{|\left[Q_k^\top\delta^k\right]_{(i)}|}{\displaystyle\sum_{j=1}^c |\left[Q_k^\top\delta^k\right]_{(j)}|} & \text{ if } i \leq c\\
1 & \text{ if } i > c,\\
\end{array}
\right. 
\end{equation}
where $c$ is the largest integer satisfying
$$0 < p_k + c - n \leq \displaystyle\sum_{j=1}^{n} \frac{|\left[Q_k^\top\delta^k\right]_{(j)}|}{|\left[Q_k^\top\delta^k\right]_{(c)}|},$$
and we have use the order statistics notation $|Q_k^\top \delta^k|_{(1)} \leq |Q_k^\top \delta^k|_{(2)} \leq \dots \leq |Q_k^\top \delta^k|_{(n)}$. 
\end{theorem}

\section{Basis Sketching}\label{sec:bask}
Having discussed all the components, we are now in a position to fully state a basis sketching model-based trust-region algorithm. Pseudocode is provided in \Cref{alg:bask}.

\begin{algorithm}[h!]
\caption{Basis Sketching Trust-Region Method}
\label{alg:bask}
\textbf{(Initialization)} Choose algorithmic constants $\eta_1, \eta_2, \Delta_{\max} > 0$
and $0 < \nu_1 < 1 <\nu_2$. \\
Choose initial point $x^0\in\Reals^n$ and initial trust-region radius $\Delta_0 \in (0, \Delta_{\max})$. \\
Initialize $\bar{g}^0\in\Reals^n$.\\%\footnote{ideally with the gradient of a fully linear model on $\cB(x^0, \Delta_0)$ of $f$, but this can be arbitrary}. 
Initialize a bank of points $\cY$ with pairs $(x,f(x))$ for which $f(x)$ is known.\\
\For{$k=1,2,\dots$}
{
\textbf{(Get initial subspace)} Use \Cref{alg:subspace} to obtain $S_k$, $S_k^\perp$ and $Q_k$.\\
\textbf{(Choose sketch size and error estimate)} Choose $p_k$ and $\delta^k$. \label{line:choices}\\
\textbf{(Determine probabilities)} Compute $\pi^k$ according to \eqref{eq:prob_dist}.\\
\textbf{(Realize a random subset)} Generate $J_k$ using Bernoulli parameters $\pi^k$. \\
\textbf{(Perform additional function evaluations)} Evaluate $\{f(x^k + \Delta_k q^k_i):i\in J_k\}$ and update $\cY, S_k$ and $S_k^\perp$.\\
\textbf{(Choose interpolation set)} Use \Cref{alg:interpolation_set} to obtain $Y_k$. \\
\textbf{(Get model parameters)} Compute model gradient $\hat{g}^k$ and model Hessian $H^k$ from \eqref{eq:bask}. \\
\textbf{(Compute ameliorated estimator)} Compute $\tilde{g}^k$ via \eqref{eq:tildegk2}. \label{line:ameliorated}\\
\textbf{(Update average estimator)} Update $\bar{g}^k$ via \eqref{eq:closed-form-g} \label{line:barg}\\
\textbf{(Solve TRSP)} (Approximately) solve $\displaystyle\min_{y \in \Reals^n} m_k(y)\triangleq \tilde{g}^{k \top} y + \frac{1}{2}y^\top H^k y$ to obtain $d^k$. \label{line:trsp}\\
\textbf{(Evaluate new point)} Evaluate $f(x^k + d^k)$ and update $\cY$. \\
\textbf{(Determine acceptance)} Compute $\rho_k\gets \displaystyle\frac{f(x^k) - f(x^k+d^k)}{m_k(0) - m_k(d^k)}$. \\
 \eIf{$\rho_k\geq\eta_1$}{
$x^{k+1}\gets x^k + d^k$.}
%\Else
{$x^{k+1}\gets x^k$.}
\textbf{(Trust-region adjustment)} \eIf{$\rho_k\geq\eta_1$}{
%\label{line:tr_adjust}}{
\eIf{$\|\tilde{g}^k\|\geq \eta_2\delta_k$}{
$\Delta_{k+1}\gets \min\{\nu_2\Delta_k,\Delta_{\max}\}.$}
%Else
{$\Delta_{k+1} \gets \nu_1\Delta_k$}
}
%\Else
{
$\Delta_{k+1}\gets\nu_1\Delta_k$.
}
}
\end{algorithm} 

Summarily, the basis sketching method begins each iteration employing a method very much resembling Algorithm~4.1 in \cite{SW08}.
This method chooses an orthogonal matrix $Q_k$ such that a set of previously evaluated points in $Y_k$ satisfies the first condition of \Cref{thm:sfullylinear} for a sketching matrix defined by the transpose of the first few columns of $Q_k$.  
Pseudocode for this initial subspace-determining algorithm is stated in \Cref{alg:subspace}. 

\begin{algorithm}[h!]
\caption{Identify Initial Subspace}
    \label{alg:subspace}
    \textbf{Input:} Center point $x\in\Reals^n$, bank of evaluated points $\cY = \{(y^1,f(y^1),\dots,(y^{|\cY|},f(y^{|\cY|}))\}$ satisfying $(x, f(x))\in\cY$, trust region radius $\Delta$. \\
    \textbf{Initalize: } Choose algorithmic constants $c\geq 1$, $\theta_1\in(0,\frac{1}{c}]$.\\
    Set $S = \{s^1\} = \{0_n\}$. \\
    Set $S^\perp = I_n$. \\
    \For{$i=1,\dots,|\cY|$}{
        \If{$\|y^i-x\|\leq c\Delta$ and $\left|\text{proj}_{S^\perp}\left(\frac{1}{c\Delta}(y^i-x)\right)\right|\geq\theta_1$}
        {
            $S = S \cup \{y^i-x\}$\\
            Update $S^\perp$ to be an orthonormal basis for $\mathcal{N}([s^2 \cdots s^{|S|}])$\\
        }
        \If{$|S| = n + 1$}{
        \textbf{break} (the for loop)
        }
    }
    $S = [s^2, \cdots, s^{|S|}]$\\
    $Q = [S \hspace{0.5pc} S^\perp]$\\
    \textbf{Return: } $S, S^\perp, Q$ \\
\end{algorithm}

\Cref{alg:subspace} is a greedy procedure for selecting a subspace based on the bank of previously evaluated points $\cY$. 
\Cref{alg:subspace} maintains two orthogonal subspaces, $S$ and $S^\perp$, which are effectively initialized as $\{0_n\}$ and $\Reals^n$, respectively. 
If a point in the bank is both within distance $c\Delta$ from $x$, and has a sufficiently large projection onto the current subspace $S^\perp$, then its displacement from $x$ is added to a set of vectors whose span is $S$. 
We then update a(n orthonormal) basis for $S^\perp$; although not explicit in the statement of \Cref{alg:subspace}, this is achieved in practice via an insertion into a maintained QR decomposition. 

Returning to \Cref{alg:bask}, given an expected sketch size $p_k$ and some estimate $\delta_k$ of $\bar{g}^{k-1} - \Egk$ (we will discuss practical means of choosing these quantities in \Cref{sec:practical}, but state the algorithm in full generality allowing for any choice of $p_k$ and $\delta_k$), we compute a probability distribution $\pi^k$ on the columns of $Q_k$. 
We then realize a random subset $J_k$ according to $n$ independent Bernoulli variables with respective probability parameters $\pi_i^k$, and evaluate $f$ at each of $\{x^k + \Delta q^k_i : i\in J_k\}$.
Appropriately splitting $Q_k$ into $S_k$ and $S_k^\perp$ based on the output of \Cref{alg:subspace} and the subsequent realization of $J_k$, we then choose an interpolation set for use in the subproblem \eqref{eq:bask}.
We will make the choice of interpolation set via \Cref{alg:interpolation_set}.

\Cref{alg:interpolation_set} is a greedy procedure for selecting an interpolation set from the bank of points $\cY$. 
Its explanation is a bit more involved and is moved to the appendix, so as not to distract from the explanation of the basis sketching method, but the procedure is derived in such a way to ensure good properties of the solution to \eqref{eq:bask}, illustrated in \Cref{thm:want_pd}. 

Continuing our summary of \Cref{alg:bask} with an interpolation set $Y_k$ in hand, we next solve the subproblem \eqref{eq:bask} to obtain model parameters. 
While we use the optimal parameter vector $\beta^*$ from \eqref{eq:bask} ``as is" to define a model Hessian, we use $S_k^\top \alpha^*$ as $\hat{g}^k$ in order to compute $\tilde{g}^k$ according to \eqref{eq:tildegk2}, and then update $\bar{g}^k$ according to \eqref{eq:closed-form-g}.
The quadratic model used in the $k$th iteration is thus the one with its degree two monomials defined by $\beta^*$ and degree one terms monomials by $\bar{g}^k$. 
This quadratic model is then minimized over a trust region to obtain a trial step, and a standard acceptance test and trust region radius update is performed.  

\section{Practical Considerations}
\label{sec:practical}
In this section, we concern ourselves with two practical considerations, the first of which prevents \Cref{alg:bask} from being directly implemented as written. 
\subsection{Choosing $\delta_k$ in \Cref{line:choices}}\label{sec:choosing_delta}
We recall that in \Cref{line:choices} of \Cref{alg:bask}, we must compute some approximation $\delta_k$, as defined in \eqref{eq:prob_dist}. 
As defined, $\delta_k$ is not particularly easy to approximate. 
This motivates several modifications to \Cref{alg:bask}, which we now describe and for which we provide some theoretical motivation. 

We begin by making the key observation that the results concerning unbiasedness and variance of the estimator $\tilde{g}^k$, respectively in \Cref{thm:unbiased} and \Cref{thm:variance},
hold \emph{regardless of the value of the previous iteration's average estimator}, $\bar{g}^{k-1}$. 
Thus, for the sake of choosing an ameliorated estimator $\tilde{g}^{k}$ in each iteration, we will replace $\bar{g}^{k-1}$ in \eqref{eq:tildegk2} with $0_n$.
In other words, we effectively replace \eqref{eq:tildegk2} in \Cref{line:ameliorated} of \Cref{alg:bask} with
\begin{equation}
    \label{eq:tildegk_instead}
    \tilde{g}^k \gets S_k^\top D(\pi^k,I_k)S_k\hat{g}^k.
\end{equation}
Per \Cref{thm:unbiased}, \eqref{eq:tildegk_instead} is still an unbiased estimator of an approximation of $\nabla f(x)$, but it exhibits a different variance.
However, when $x^k$ approaches a stationary point, the vector $0_n$ ought to become an increasingly good initial approximation of the gradient near stationarity, and hence the variance of \eqref{eq:tildegk_instead} as an estimator of $\nabla f(x^k)$ decreases proportionally with the stationarity $\|\nabla f(x^k)\|$. 

The substitution of \eqref{eq:tildegk_instead} has multiple practical effects on the overall logic of \Cref{alg:bask}.
First, it is apparent from \eqref{eq:tildegk_instead} that the gradient $\tilde{g}^k$ exists entirely in a (weighted) subspace determined by $S_k$.
Thus, it is practically desirable to solve a lower dimensional (the rank of $S_k$) trust region subproblem in each iteration.
To determine an appropriate Hessian approximation, we record several results, which are respectively extensions of \Cref{prop:closed-form}, \Cref{thm:unbiased} and \Cref{thm:variance}.
\begin{proposition}
    \label{prop:closed-form-H}
    Denote
    \begin{equation}
    \label{eq:barH}
    \bar{H}^k := \mat\left(
\begin{array}{rl}
\displaystyle\arg\min_{\alpha\in\Reals^{n\times n}} &\frac{1}{2}\|\vec(\alpha) - \vec(\bar{H}^{k-1})\|^2\\
 \text{s. to} & S_k \alpha S_k^\top = \hat{H}^{J_k}\\
\end{array}\right)
\end{equation}
The subproblem \eqref{eq:barH} admits a closed-form solution given by 
\begin{equation}
\label{eq:closed-form-H}
\bar{H}^k = \bar{H}^{k-1} - S_k^\top S_k \bar{H}^{k-1} S_k^\top S_k + S_k^\top\hat{H}^{J_k} S_k.
\end{equation}
\end{proposition}

\begin{proof}
    The reasoning is essentially the same as in the proof of \Cref{prop:closed-form}. 
    The KKT conditions in \eqref{eq:barH} can be expressed as
\begin{equation*}
\begin{array}{rll}
\vec(\alpha) & = \vec(\bar{H}^{k-1}) - (S_k^\top\otimes S_k^\top)\mu & (stationarity)\\
(S_k \otimes S_k)\vec(\alpha) & = \vec(\hat{H}^{J_k}) & (primal feasibility),\\
\end{array}
\end{equation*}
Plugging the stationary condition into the primal feasibility condition,
$$(S_k \otimes S_k)\vec(\bar{H}^{k-1}) - (S_k \otimes S_k)(S_k^\top \otimes S_k^\top)\mu = \vec(\hat{H}^{J_k}).$$
Using properties of the Kronecker product and the orthonormality of $Q_k$, 
this simplifies to
$$\mu = (S_k \otimes S_k)\vec(\bar{H}^{k-1})- \vec(\hat{H}^{J_k}).$$
Plugging these Lagrange multipliers back into the stationarity condition, we obtain
\begin{equation*}
\begin{array}{ll}
\vec(\alpha) & = \vec(\bar{H}^{k-1}) - (S_k^\top\otimes S_k^\top) (S_k \otimes S_k)\vec(\bar{H}^{k-1})- (S_k^\top\otimes S_k^\top) \vec(\hat{H}^{J_k})\\
& = \vec(\bar{H}^{k-1}) - (S_k^\top S_k \otimes S_k^\top S_k) \vec(\bar{H}^{k-1})- (S_k^\top \otimes S_k^\top ) \vec(\hat{H}^{J_k})\\
& = \vec(\bar{H}^{k-1}) - S_k^\top S_k \bar{H}^{k-1} S_k^\top S_k - S_k^\top \hat{H}^{J_k} S_k,
\end{array}
\end{equation*}
where we have used the property that for appropriately sized matrices $A, B,$ and $C$, 
$$ABC = \mat((C^\top\otimes A)\vec(B)).$$
\end{proof}

We state the next two results without proof, but note that by using $\vec$ and $\mat$ operators as in the proof of \Cref{eq:closed-form-H}, the proofs are virtually the same as those of \Cref{thm:unbiased} and \Cref{thm:variance}, respectively.

\begin{theorem}
\label{thm:unbiased_H}
    Suppose for each $J_k$, we can compute $\hat{H}_{J_k}$ satisfying 
    $\|S_k^\top (\hat{H}_{J_k} - \nabla^2 f(x)) S_k\| \leq \kappa_{eH}$ for all $x\in\cB(x^k, \Delta_k)$ and for some $\kappa_{eH}\in[0,\infty)$.  
    Let
    \begin{equation}
        \label{eq:tildeHk}
        \tilde{H}^k = \bar{H}^{k-1} - S_k^\top D(\pi^k, J_k) S_k \bar{H}^{k-1} S_k^\top D(\pi^k, J_k) S_k + S_k^\top D(\pi^k, J_k) \hat{H}^{J_k} D(\pi^k, J_k) S_k.
    \end{equation}
    Then, 
    $$\|\mathbb{E}_{J_k}[\tilde{H}^k] - \nabla^2 f(x)\| \leq n\kappa_{eH}$$
    for all $x\in \cB(x^k, \Delta_k)$. 
\end{theorem}

\begin{theorem}\label{thm:variance_H}
    The variance of $\tilde{H}^k$, with respect to the distribution governing $J_k$, is
    \begin{equation}
        \label{eq:variance_H} 
        \mathbb{E}_{J_k} \left[ \left\|\tilde{H}^k - \mathbb{E}_{J_k}\left[\tilde{H}^k\right]\right\|^2\right] = \left\|\bar{H}^{k-1} - \mathbb{E}_{J_k}\left[\tilde{H}^k\right]\right\|^2_{Q_kD_kQ_k^\top - I_n}.
    \end{equation}
\end{theorem}

We note that we could extend \Cref{def:s_full_linearity} to produce a definition of $(S,\kappa_{ef},\kappa_{eg}, \kappa_{eH})$-fully quadratic models of $f(x)$.
Doing so would  enable a $\mathcal{O}(\Delta_k)$ error bound (as opposed to $\mathcal{O}(1)$) in \Cref{thm:unbiased_H}.
However, because our model-building procedure employs the subproblem \eqref{eq:bask} and uses \Cref{alg:interpolation_set} to determine an interpolation set that only guarantees an affinely independent subset of points, our algorithm is not intended to guarantee fully quadratic models in a subspace. 
Thus, such a fully quadratic extension is outside of the scope of this paper, but would be easily achieved. 
We also note a high-level similarity between the motivation for \eqref{eq:barH} and what is done in work on randomized Hessian estimation in \cite{leventhal2011randomized}. 

Taken together, \Cref{prop:closed-form-H}, \Cref{thm:unbiased_H} and \Cref{thm:variance_H} imply that, given the subspace gradient $D(\pi^k,J_k)S_k\tilde{g}^k \in \Reals^{rank(S_k)}$ implied by \eqref{eq:tildegk_instead}, 
and provided we maintain an average estimator $\bar{H}^{k-1}$ between iterations, 
a reasonable choice of corresponding model Hessian is the projection of the unbiased estimator
$\tilde{H}^k$ into the subspace defined by $S_k$, that is, the model Hessian is
$S_k\tilde{H^k}S_k^\top\in\Reals^{rank(S_k)\times rank(S_k)}$. 

To summarize these changes to \Cref{alg:bask}, we
\begin{itemize}
    \item Replace the ameliorated estimator update \eqref{eq:tildegk2} with \eqref{eq:tildegk_instead} in \Cref{line:ameliorated}
    \item Maintain an average estimator for the Hessian $\bar{H}^{k}$ via the update \eqref{eq:barH} in addition to the average estimator for the Hessian $\bar{g}$ in \Cref{line:barg}.
    \item Replace the subproblem in \Cref{line:trsp} with the lower dimensional subproblem 
    \begin{equation}
        y^* := \displaystyle\min_{y\in\Reals^{rank(S_k)}} \langle D(\pi^k,I_k)S_k\hat{g}^k, y\rangle  + \frac{1}{2}y^\top S_k\tilde{H^k}S_k^\top y,
    \end{equation}
    where $\tilde{H}^k$ is computed as in \eqref{eq:tildeHk}.
    The trial step is replaced with $d^k = S_k^\top y^*$.
\end{itemize}

With these changes made, a reasonable estimator for $\delta_k$ is simply $\bar{g}^k$, the (biased) estimate of a $\mathcal{O}(\Delta_k)$-accurate approximation to $\nabla f(x^k)$.

\subsection{Choosing $p_k$ in \Cref{line:choices}}
While we do not intend to prove convergence results for \Cref{alg:bask} in this paper, we note that \Cref{alg:bask} can be readily analyzed as a trust region method with probabilistic models and deterministic function values, which is within the scope of \cite{Gratton2017complexity}. 
Indeed, by viewing \Cref{alg:bask} through the lens of probabilistic models, we will derive a practical method for dynamically selecting $p_k$. 
In light of \Cref{thm:variance}, larger expected (and hence, realized) sketch sizes $p_k$ will intuitively yield estimators of lower variance, suggesting more stable convergence to a local minimizer of $f$.
In one extreme, and by our prior observations, the variance of the estimator $\tilde{g}^k$ is 0 when $p_k=n$; in this case, \Cref{alg:bask} is just a deterministic DFO trust-region method. 
We now state a result about the probabilistic accuracy of our models, using conservative (Markov inequality-derived) concentration inequalities. 

\begin{theorem}
\label{thm:prob_fl}
Let \Cref{ass:cd} hold. 
Suppose in the $k$th iteration of \Cref{alg:bask}, we denote
$$V:= \|\bar{g}^{k-1} - \mathbb{E}_{I_k}[\tilde{g}^k]\|^2_{Q_k D_k Q_k^\top - I_n},$$
the variance of $\tilde{g}^k$. 
Let $C > 0$ and let $\kappa_{eg}$ be as in \Cref{thm:sfullylinear}. 
If 
$V < n\kappa_{eg}^2\Delta_k^{2}$,
then, 
$$\|\tilde{g}^k - \nabla f(x^k)\| \leq (C+1)\sqrt{n}\kappa_{eg}\Delta_k$$
with probability $1 - \frac{1}{C^2}$
\end{theorem}

\begin{proof}
By Chebyshev's inequality and the supposition on $V$,
$$\|\tilde{g} - \mathbb{E}_{I_k}[\tilde{g}^k]\| \leq C\sqrt{V} < C\sqrt{n}\kappa_{eg}\Delta_k$$
with probability $1-\frac{1}{C^2}$.
Combining this result with \Cref{thm:unbiased}, 
we get the desired result.
\end{proof}

In words, \Cref{thm:prob_fl} shows that the models defined by $\tilde{g}^k$ nearly satisfy the definition of being probabilistically fully linear as defined by, for instance, \cite{Gratton2017complexity, Cartis2015, Chen2017, BCMS2018}, when the variance $V$ is sufficiently small. 
Thus, given the approximation of $V$ (the quality of which is, of course, totally dependent on the quality of the estimate $\delta^k$), \Cref{thm:prob_fl} suggests employing the adaptive scheme in \Cref{alg:choose_p} for choosing a sketch size $p_k$.
In \Cref{alg:choose_p}, the constant $C$ has effectively absorbed the unknown constant $\kappa_{eg}$.
Choosing $C=0$ will force the estimator $\tilde{g}^k$ to have zero variance (that is, we will choose each column of $Q_k$ with probability one), while larger values of $C$ will result in the estimator $\tilde{g}^k$ exhibiting proportionally more variance. 

\begin{algorithm}[h!]
\caption{\label{alg:choose_p} Adaptive scheme to choose $p_k$}
\textbf{Input: } Trust-region radius $\Delta_k$, orthonormal matrix $Q_k$, accuracy parameter $C\geq 0$, error estimate $\delta^k$, and minimal expected size of sample $b_0\in(0,n)$. \\
$b \gets b_0$. \\
\While{$b\leq n$}{
Compute $\pi^k$ according to \eqref{eq:opt_prob} with $p_k=b$ and $\delta^k$.\\
\eIf{$\|\tilde\delta^k\|_{Q_k D_k Q_k^\top - I_{n}} \leq nC^2\Delta_k^2$}{
\textbf{return} $p_k = b$.}
{
$b\gets b+1$.\\
}
}
\end{algorithm}
%We note that \Cref{thm:sfullylinear} explicitly gives us a reasonable bound on $\kappa_{eg}$, if we can coarsely estimate an upper bound on $L_g$ and $L_{mg}$. 

\subsection{A note on the realization of $J_k$}
In this paper, for simplicity, our subproblem for determining \eqref{eq:prob_dist} only constrains the \emph{expected} size of a realized sample $|J_k|$.
In settings where function evaluations are particularly expensive, it may be very undesirable to under-utilize (or worse, over-utilize) available computational resources by not being able to control the actual, realized number of function evaluations performed in each iteration of \Cref{alg:bask}.
In \cite{SAMP2022}, we proposed a means to handle this via conditional Poisson sampling (see, e.g., \cite{chen1994weighted,chen2000general}). However, in this paper, we do not concern ourselves with this issue. Incorporating conditional Poisson sampling in the present work would be straightforward. 

\section{Numerical Results}\label{sec:numerical}
We implemented \Cref{alg:bask} using the existing \texttt{Matlab} software for \texttt{POUNDers} \footnote{The implementation of \texttt{POUNDers} is available as part of the Practical Optimization Using Structure (POptUS) repos, and is specifically available at \url{https://github.com/POptUS/IBCDFO}} as a starting point. 
In particular, the same set of parameters dictating trust-region dynamics are used in both methods, the trust-region subproblems are solved in the same way, and as described previously, our implementations of \Cref{alg:subspace} and \Cref{alg:interpolation_set} are very minor modifications of a model-building routine (\texttt{formquad.m}) that already existed in \texttt{POUNDers}.

While \texttt{POUNDers} was originally developed for nonlinear least squares problems \eqref{eq:nonlsq}, the most recent implementation allows for more general problems of the form \eqref{eq:unconstrained}. 
As an extension of this implementation of \texttt{POUNDers}, our implementation of \Cref{alg:bask}, which we call \texttt{SS-POUNDers}, or ``subspace \texttt{POUNDers}", can solve problems of both the general form \eqref{eq:unconstrained} as well as the specialized nonlinear least squares form \eqref{eq:nonlsq}. 
\texttt{POUNDers} (and hence, \texttt{SS-POUNDers}) achieves this by maintaining separate quadratic models of each component function $f_i(x)$, each parameterized by a component model gradient term $g_i\in\Reals^n$ and a component model Hessian term $H_i\in\Reals^{n\times n}$. 
In turn, these component models are combined to yield a full model gradient $g(x)$ and full model Hessian $H(x)$ via
\begin{equation}
\label{eq:fullspace}
g(x) = \displaystyle\sum_{i=1}^p f_i(x)g_i, \quad H(x) = \displaystyle\sum_{i=1}^p \left[g_i g_i^\top + H_i\right].
\end{equation}
By applying the average estimator update \eqref{eq:closed-form-g} to each component model gradient $g_i$ and the update \eqref{eq:closed-form-H} to each component model Hessian $H_i$, it is clear by linearity that we are effectively maintaining average estimators  of $g(x)$ and $H(x)$ in \eqref{eq:fullspace}. 

\subsection{Parameter Settings}
We use all the same default parameters as in \texttt{POUNDers}. 
In particular, in \Cref{alg:bask}, we use $\eta_1 = 0.05, \Delta_{\max} = 1000\Delta_0, \nu_1 = 0.5,$ and $\nu_2 = 2.0$.
In \Cref{alg:subspace}, which is again derived from the model-building routine of \texttt{POUNDers}, we use the \texttt{POUNDers} default settings of $c = \sqrt{n}$ and $\theta_1 = 10^{-5}$. 
Similarly, in \Cref{alg:interpolation_set}, we use the \texttt{POUNDers} defaults of $c = \sqrt{n}$ and $\theta_2 = 10^{-3}.$
In terms of parameters unique to \texttt{SS-POUNDers}, in \Cref{alg:bask}, we chose
$\eta_2 = 10^{-3}$ and in \Cref{alg:choose_p}, we chose $C=0.01\sqrt{n}$ and $b_0 = 1$. 

Finally, we make the observation that \Cref{alg:bask} intentionally leaves the choice of $\bar{g}^0$ open.
Naturally, if one had knowledge of a reasonable estimate of $\nabla f(x^0)$ (from previously obtained function evaluations, for instance), then one should employ that. 
In the assumed absence of that information, and in our tests, we tried two natural, but very different, choices of $\bar{g}^0$.
Our first choice was to proceed like \texttt{POUNDers} in the first iteration and simply spend $n+1$ function evaluations to compute an initial ($\mathcal{O}(\Delta_0)$-accurate) simplex gradient. 
The obvious disadvantage to this approach is that it spends $n$ function evaluations immediately.
The advantage to this approach is that despite our amendments to \Cref{alg:bask} discussed in \Cref{sec:choosing_delta}, it is still important to maintain a good approximation $\bar{g}^k$ for the sake of obtaining reasonable variance estimates, and hence sample sizes $p_k$, from \Cref{alg:choose_p}. 
Beginning the algorithm with a reasonably accurate $\bar{g}^0$ promises to lower the variance of estimators throughout the course of the algorithm.
Our second choice naively set $\bar{g}^0 = 0_n$. 
This second choice has completely opposite advantages and disadvantages to the first choice.
We pay no upfront cost to obtaining an ``approximation" to $\nabla f(x^0)$, but the approximation is completely arbitrary and our variance estimates will accordingly be arbitrary, likely for many iterations. 
We found in preliminary testing that the former choice was practically superior, and so we implement that in \texttt{SS-POUNDers}, and demonstrate only the results stemming from that choice in what follows. 

\subsection{Test Problems}
We employ two separate test sets, the Mor{\'e}-Wild testset as implemented in the ``Benchmarking DFO" (BenDFO) repoistory \cite{BENDFOCode}, and ``Yet Another Test Set for Optimization" (YATSOP) repository \cite{YATSOPCode}. 
Both test sets consist of unconstrained nonlinear least squares problems of the form \eqref{eq:nonlsq}.
The former test set, which we simply refer to as BenDFO in the remainder, 
is better-known and was first compiled from (mostly) extant problems in \cite{JJMSMW09}. 
For our purposes, we note that these 53 problems are all fairly low-dimensional, with $n$ ranging between 2 and 12. 
The latter test set, which we simply refer to as YATSOp, consists of 38 problems and are significantly larger in dimension, with $n$ ranging between 98 and 125. 
We chose all of the problems labelled ``midscale" in YATSOp. 

The decision to employ these two separate test sets was based on two considerations.
First is the fact that they have both been previously employed in DFO literature; BenDFO has been used very frequently, and YATSOp most recently in \cite{STARS2022}, with a superset of problems similar to YATSOp employed in \cite{cartis2019derivative, CRsubspace2021}. 
Secondly, the difference in average problem sizes between the two test sets is important for this study. 
Our intention is to demonstrate that for small problems (as represented in BenDFO), the performance of \texttt{SS-POUNDers} is \emph{comparable} to the performance of \texttt{POUNDers}. 
In low dimensions $n$, one intuitively expects that finding $n + 1$ affinely independent interpolation points is ``more likely". 
That is, potentially expensive geometry-improvement steps in \texttt{POUNDers} are less likely to occur. 
However, as $n$ increases, one expects (and often sees in practice) that finding $n+1$ affinely independent points is ``less likely", and \texttt{POUNDers} will therefore require geometry improvement steps more often. 
\texttt{SS-POUNDers}, on the other hand, does not have a geometry improvement step like \texttt{POUNDers}, because it only aims to ensure $S_k$-full linearity in each iteration, where $S_k$ was initially chosen by \Cref{alg:subspace} and then randomly augmented with points that are well-poised in the nullspace of the initial subspace chosen by \Cref{alg:subspace}. 
Based on this intuition, we expect that it is more likely to see \texttt{SS-POUNDers} outperform \texttt{POUNDers} on YATSOp than BenDFO, but we would hope that \texttt{SS-POUNDers} at least recovers the performance of \texttt{SS-POUNDers}.

\subsection{Performance Profiles}
We will use performance profiles (see, e.g., \cite{JJMSMW09}) to illustrate the relative performance of \texttt{POUNDers} and \texttt{SS-POUNDers}. 
Given a tolerance $\tau>0$, a set of solvers $\mathcal{S}$ and a set of problems $\mathcal{P}$, we count, for each combination of solver and problem, the number (normalized by $n+1$, where $n$ is the the problem dimension) $N_{solv,prob,\tau}$ of evaluations of $f(x)$ a solver $solv$ must make on problem $prob$ before it evaluates an $x$ such that 
$$f(x) \leq f(x^*) + \tau(f(x^0) - f(x^*)).$$
Remarkably, despite the nonconvexity of most of the problems in BenDFO and YATSOp, but owing to the original curation of these test sets, every run converges to a neighborhood of the known optimal solution $x^*$, and so the metric $N_{solv,prob,\tau}$ is reasonable for these tests.
%Moreover, this choice of $N_{solv,prob,\tau}$ is motivated by the fact that in many applications of nuclear model calibration, the impetus behind this work, there are often error bars on observable calculations, and so a nuclear physicist often requires only a fixed, predetermined number of digits of accuracy on the output of a model calibration. 

A performance profile computes, for each problem $prob$,
$$N_{prob,\tau} = \displaystyle\min_{solv\in\mathcal{S}} N_{solv,prob,\tau}$$
and then plots $\alpha\geq 1$ on the $x$-axis and
$$\frac{1}{|\mathcal{P}|}\left|\left\{prob\in\mathcal{P} : N_{solv,prob,\tau} \leq\alpha N_{prob,\tau}\right\} \right|$$
on the $y$-axis for each solver $solv\in\mathcal{S}$.
In words, a performance profile shows, as a function of $\alpha$, the percentage of problems solved to tolerance $\tau$ by each solver within a number of budget units no more than $\alpha$ times larger than the least number of budget units required by any solver on the same problem. 

We run \texttt{SS-POUNDers} a total of 30 times on each problem, each with a different random seed. 
In the profiles in \Cref{fig:bendfoperfprofs}, we illustrate the \emph{median} performance of \texttt{SS-POUNDers} on BenDFO.
\begin{figure}
\caption{\label{fig:bendfoperfprofs} Performance profiles comparing the performance of \texttt{POUNDers} with the \emph{median} performance of \texttt{SS-POUNDers} on the (low-dimensional) BenDFO test set with convergence tolerances $\tau=10^{-1}$ (left), $\tau=10^{-3}$ (center), and $\tau=10^{-5}$ (right).}
    \centering
    \includegraphics[width=.3\linewidth]{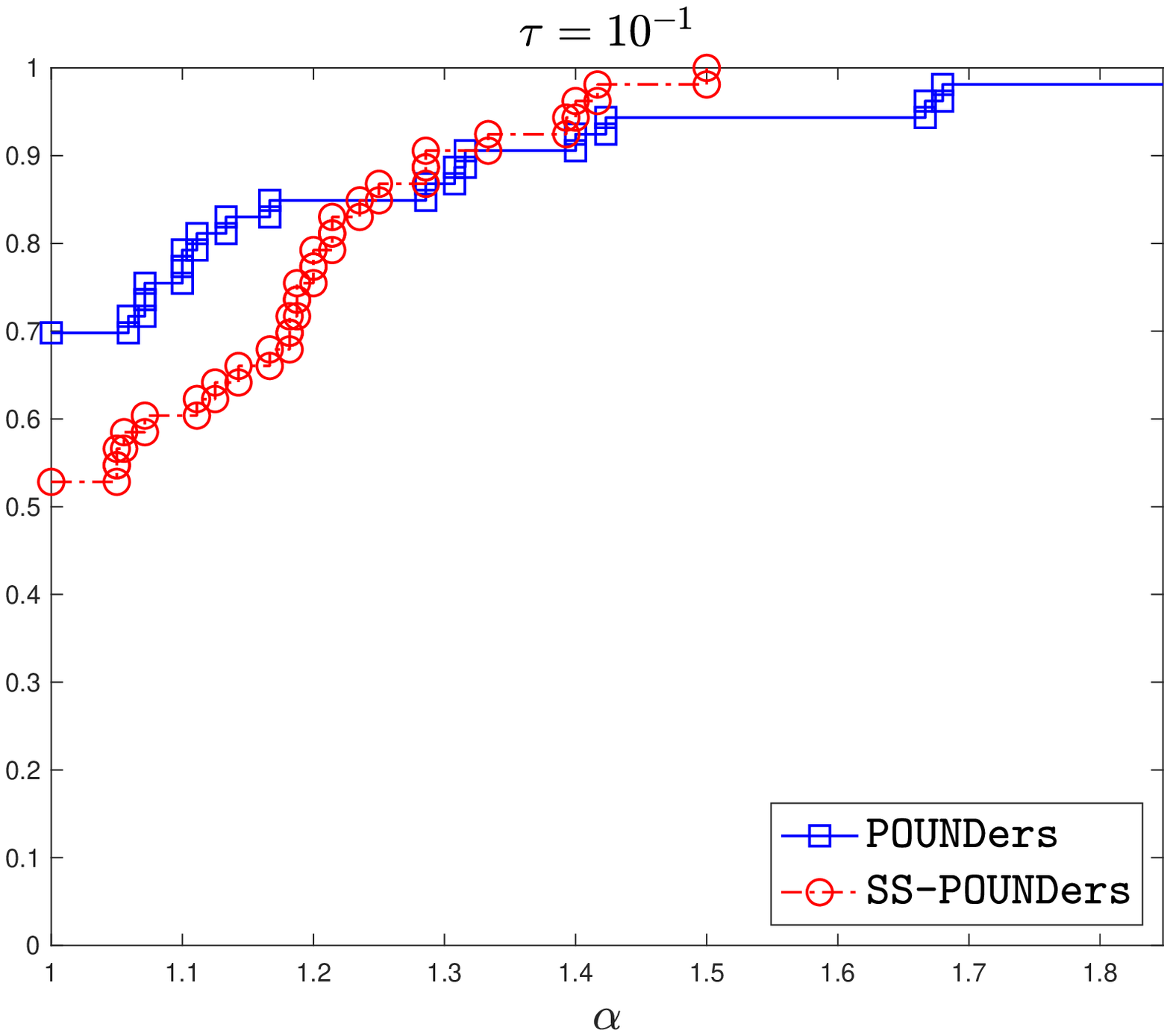}
    \includegraphics[width=.3\linewidth]{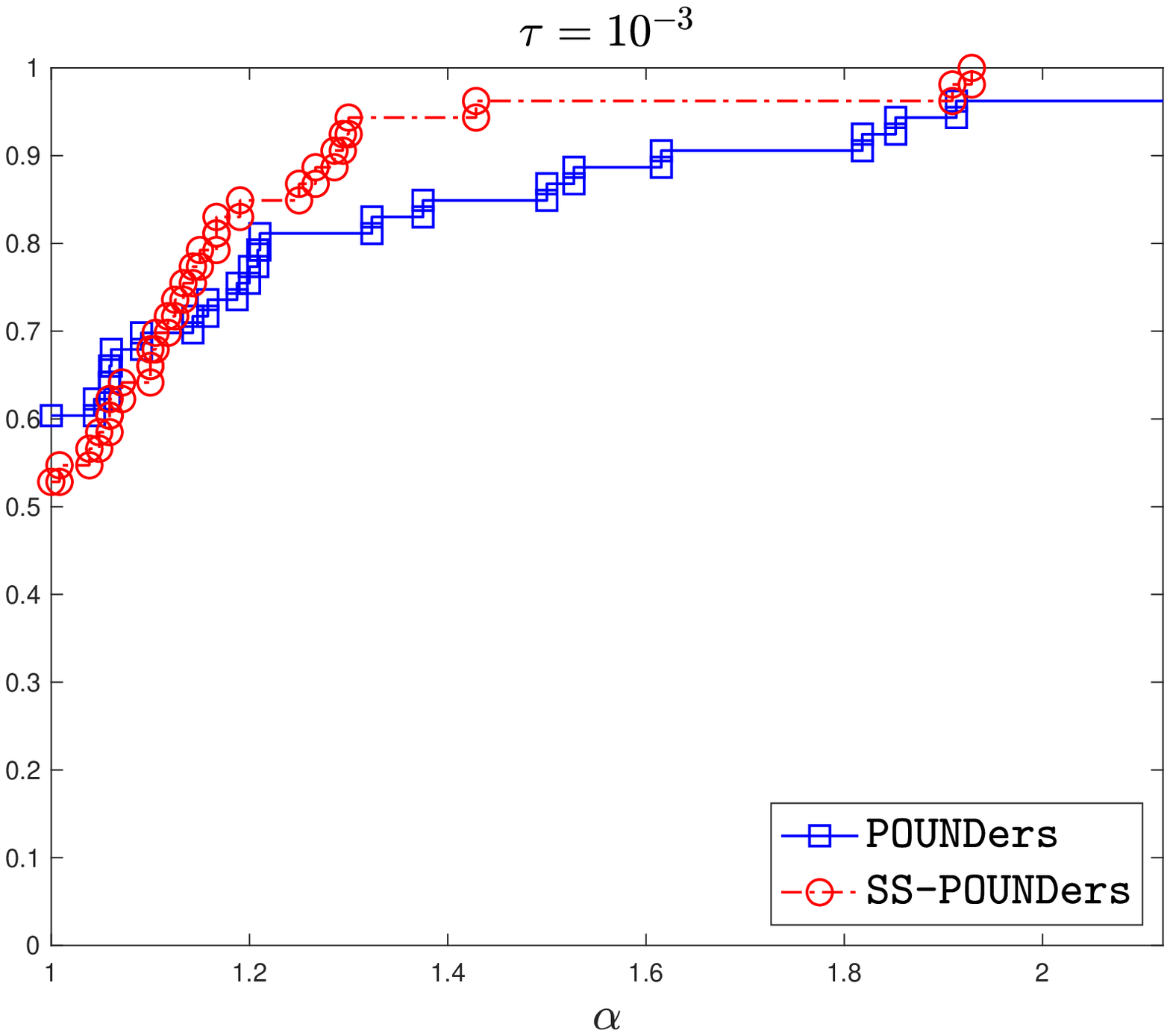}
    \includegraphics[width=.3\linewidth]{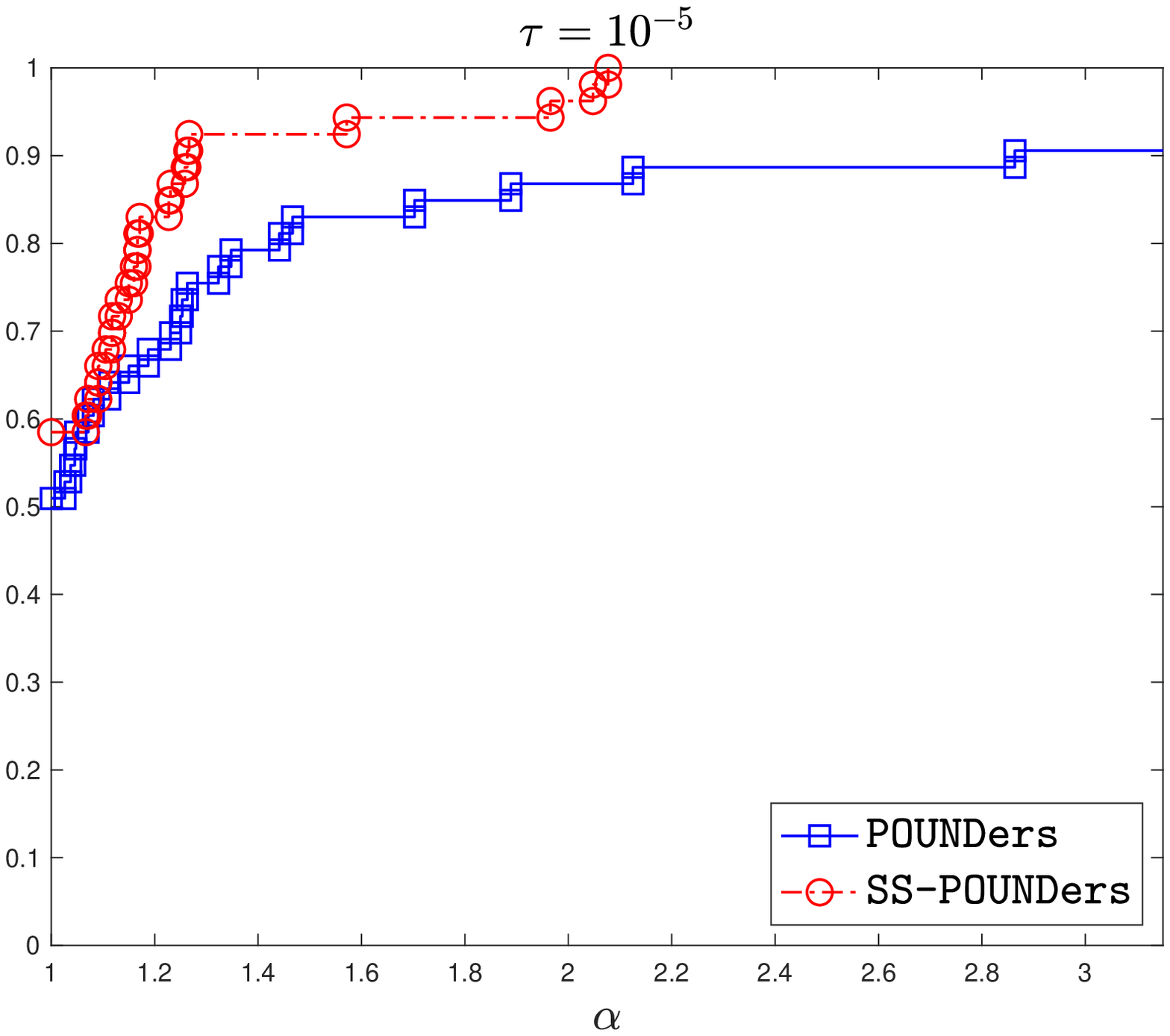}
\end{figure}

In the profiles in \Cref{fig:yatsopperfprofs}, we illustrate the median performance of of \texttt{SS-POUNDers} on YATSOp. 
In both BenDFO and YATSOp, we see that as the convergence tolerance $\tau$ becomes tighter, there is an argument for an increasing preference for \texttt{SS-POUNDers}. 
In fact, when we demand a tight convergence tolerance ($\tau=10^{-5}$), these results demonstrate that the median performance of \texttt{SS-POUNDers} completely dominates the performance of \texttt{POUNDers.}
Although not a perfect explanation, one intuition for this phenomenon may be that the longer an algorithm runs (and hence becomes closer to convergence), the more opportunities \texttt{SS-POUNDers} is presented to randomize, and hence avoid the geometry improvement steps that \texttt{POUNDers} must take.

\begin{figure}
\caption{\label{fig:yatsopperfprofs} Performance profiles comparing the performance of \texttt{POUNDers} with the \emph{median} performance of \texttt{SS-POUNDers} on the (higher-dimensional) YATSOp test set with convergence tolerances $\tau=10^{-1}$ (left), $\tau=10^{-3}$ (center), and $\tau=10^{-5}$ (right).}
    \centering
    \includegraphics[width=.3\linewidth]{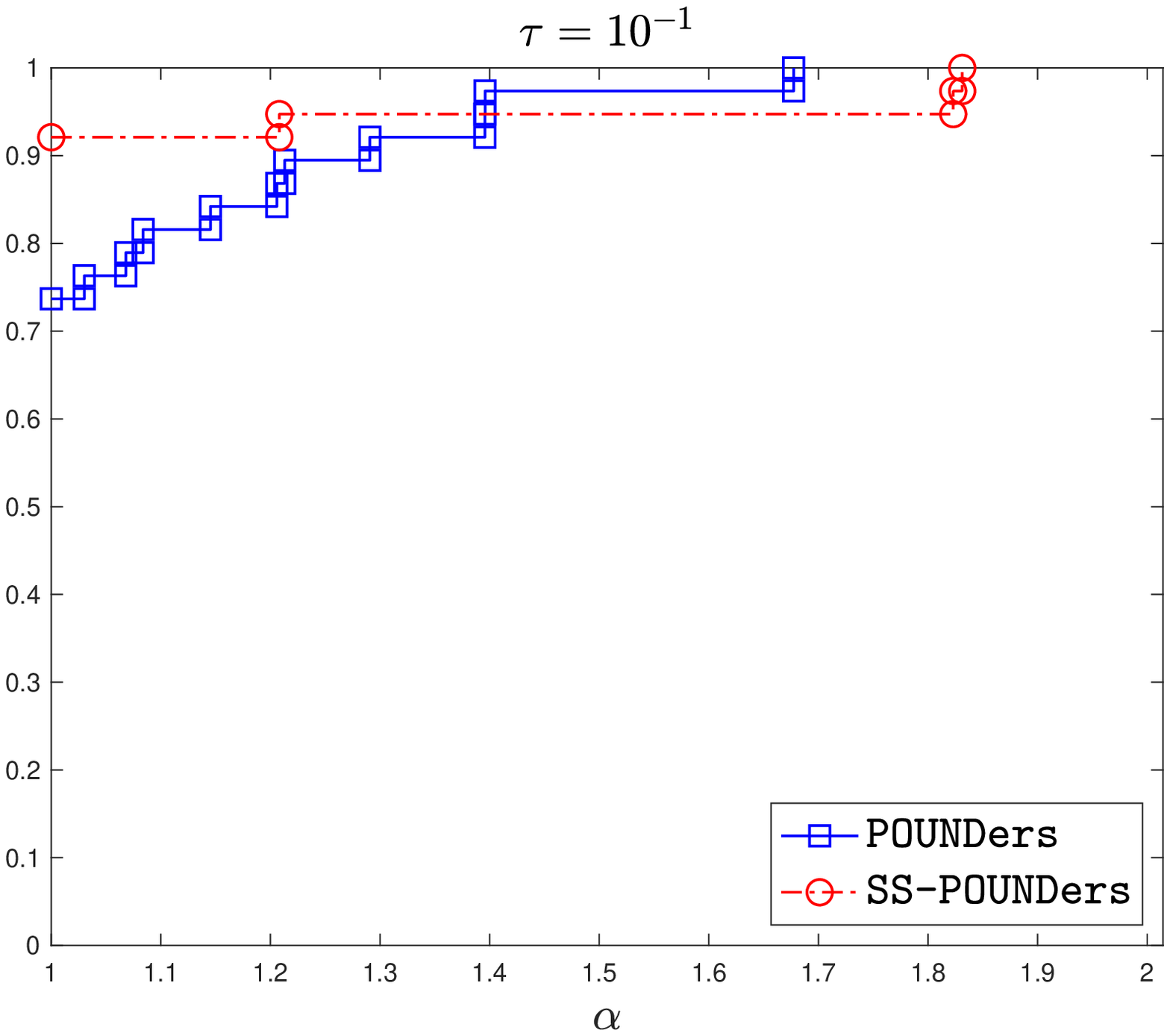}
    \includegraphics[width=.3\linewidth]{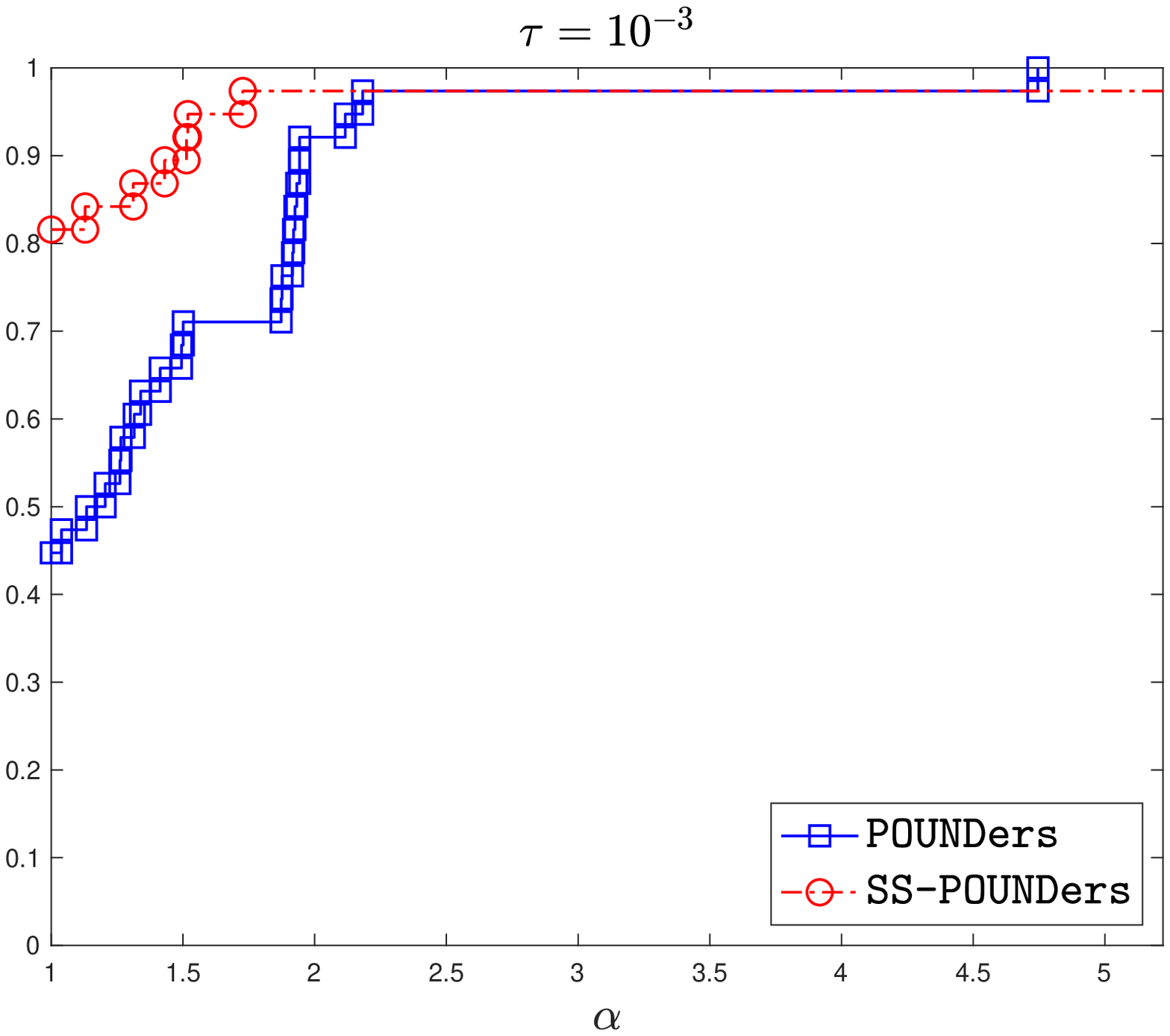}
    \includegraphics[width=.3\linewidth]{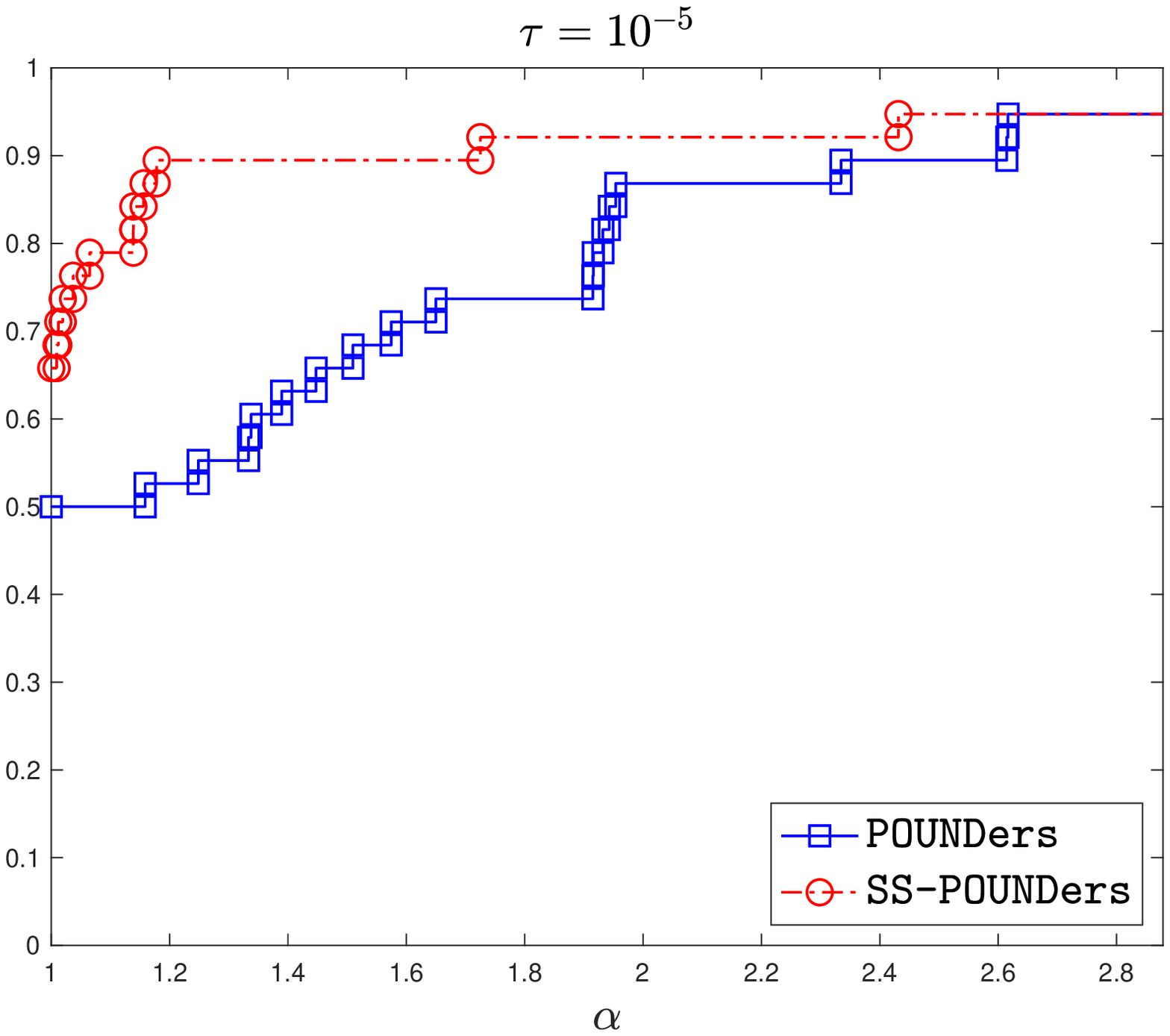}
\end{figure}

\Cref{fig:bendfoperfprofs} and \Cref{fig:yatsopperfprofs} are certainly encouraging, but because \texttt{SS-POUNDers} is a \emph{randomized} method, we must demonstrate that performance is reasonable even in the worst case, in order to demonstrate the robustness of a randomized approach.
Towards that end, we use the same data as in the previous figures, but show the \emph{worst-case} performance of \texttt{SS-POUNDers} across all 30 trials performed. 
These results are shown in \Cref{fig:worstbendfo} and \Cref{fig:worstyatsop}. 

\begin{figure}
\caption{\label{fig:worstbendfo} Performance profiles comparing the performance of \texttt{POUNDers} with the \emph{worst-case} performance of \texttt{SS-POUNDers} on the (low-dimensional) BenDFO test set with convergence tolerances $\tau=10^{-1}$ (left), $\tau=10^{-3}$ (center), and $\tau=10^{-5}$ (right).}
    \centering
    \includegraphics[width=.3\linewidth]{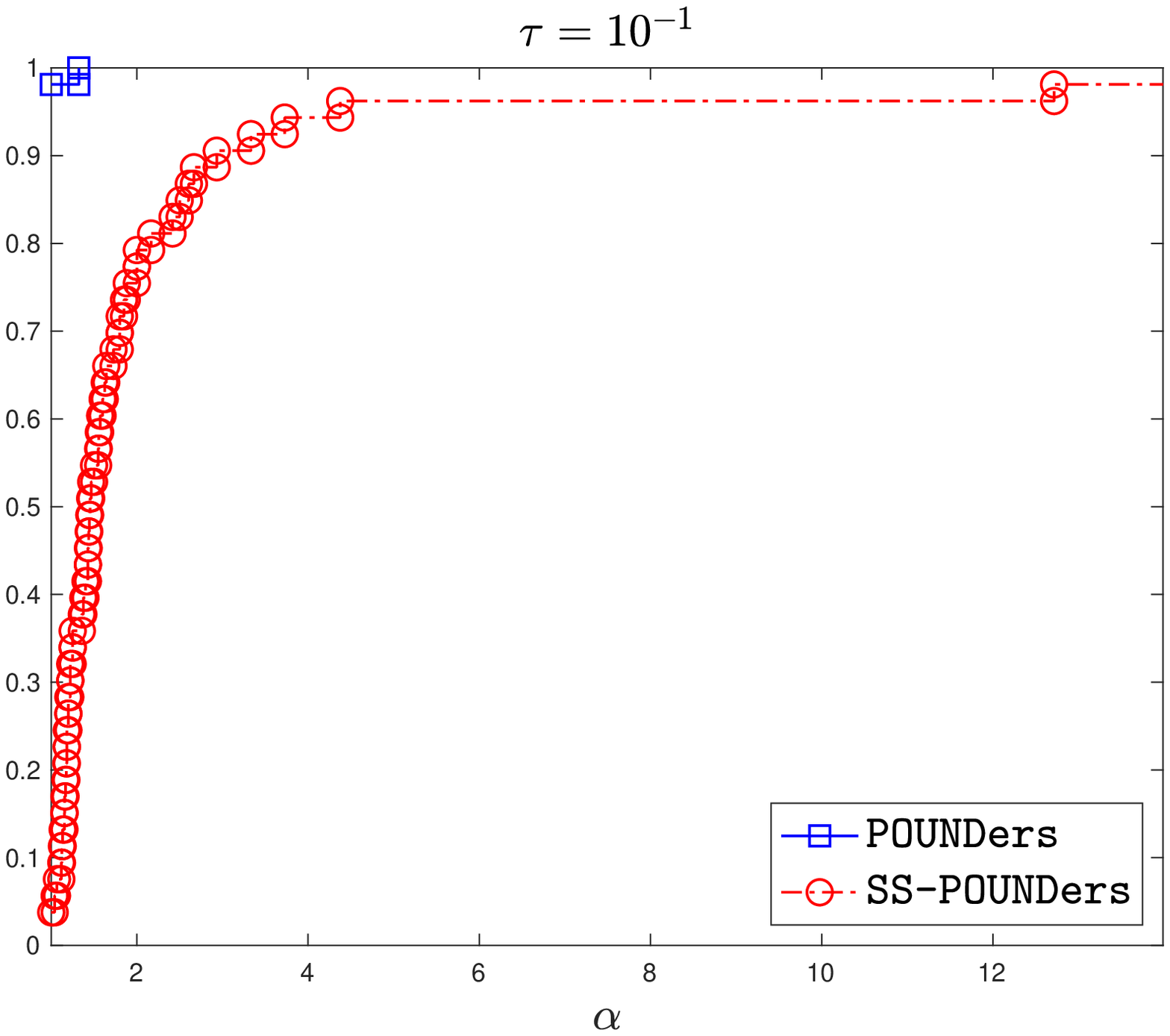}
    \includegraphics[width=.3\linewidth]{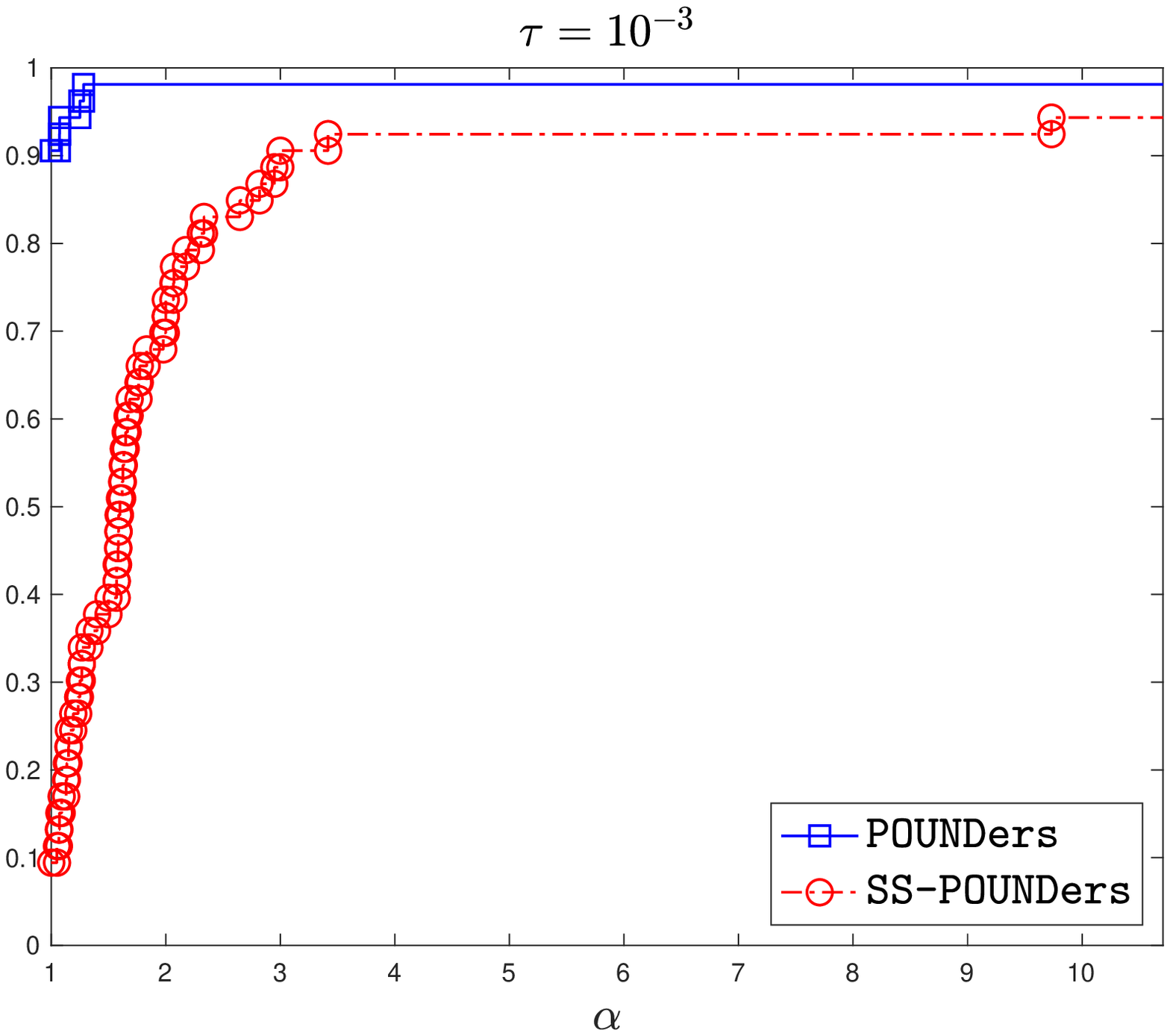}
    \includegraphics[width=.3\linewidth]{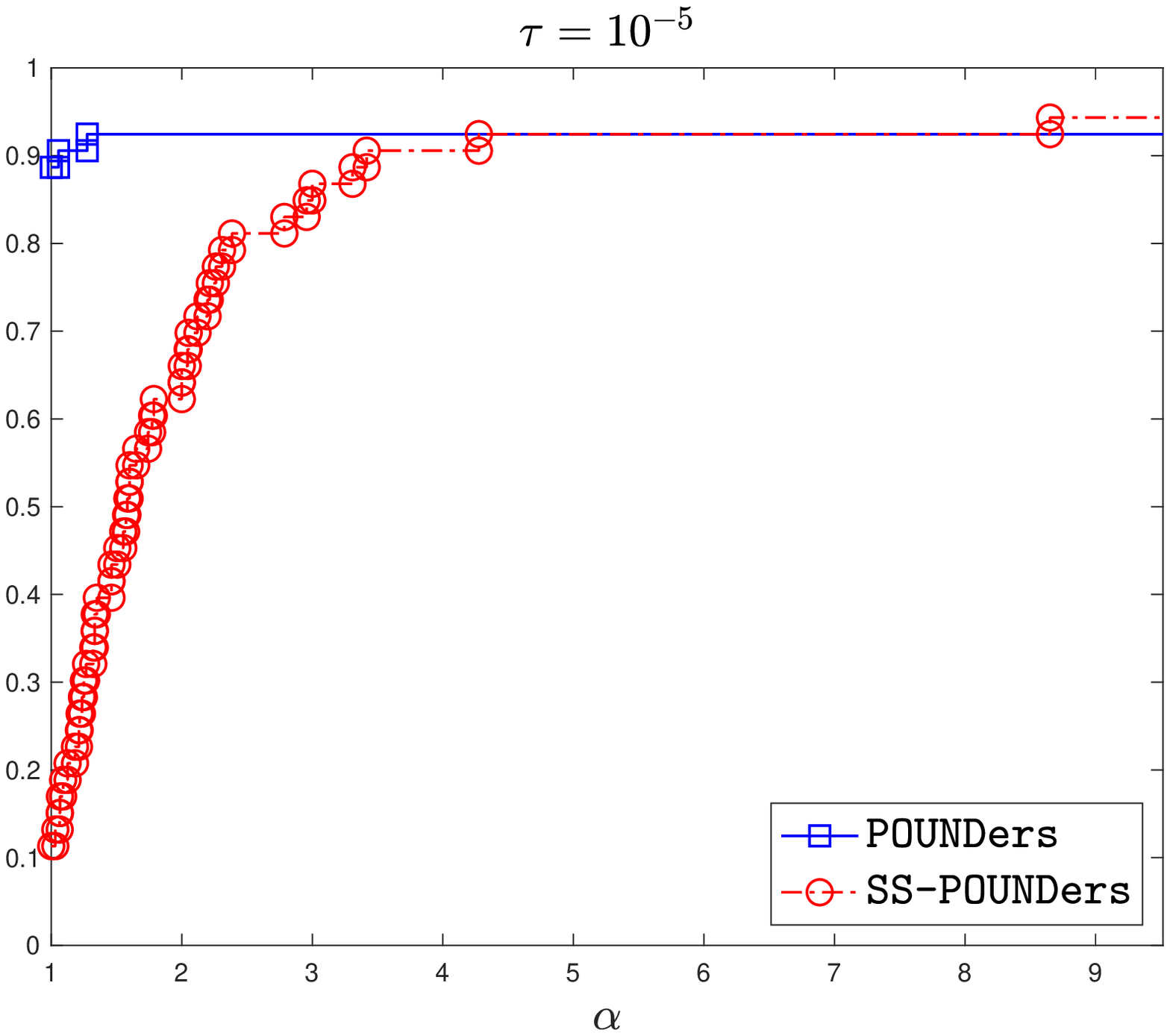}
\end{figure}

\begin{figure}
\caption{\label{fig:worstyatsop} Performance profiles comparing the performance of \texttt{POUNDers} with the \emph{worst-case} performance of \texttt{SS-POUNDers} on the (higher-dimensional) YATSOp test set with convergence tolerances $\tau=10^{-1}$ (left), $\tau=10^{-3}$ (center), and $\tau=10^{-5}$ (right).}
    \centering
    \includegraphics[width=.3\linewidth]{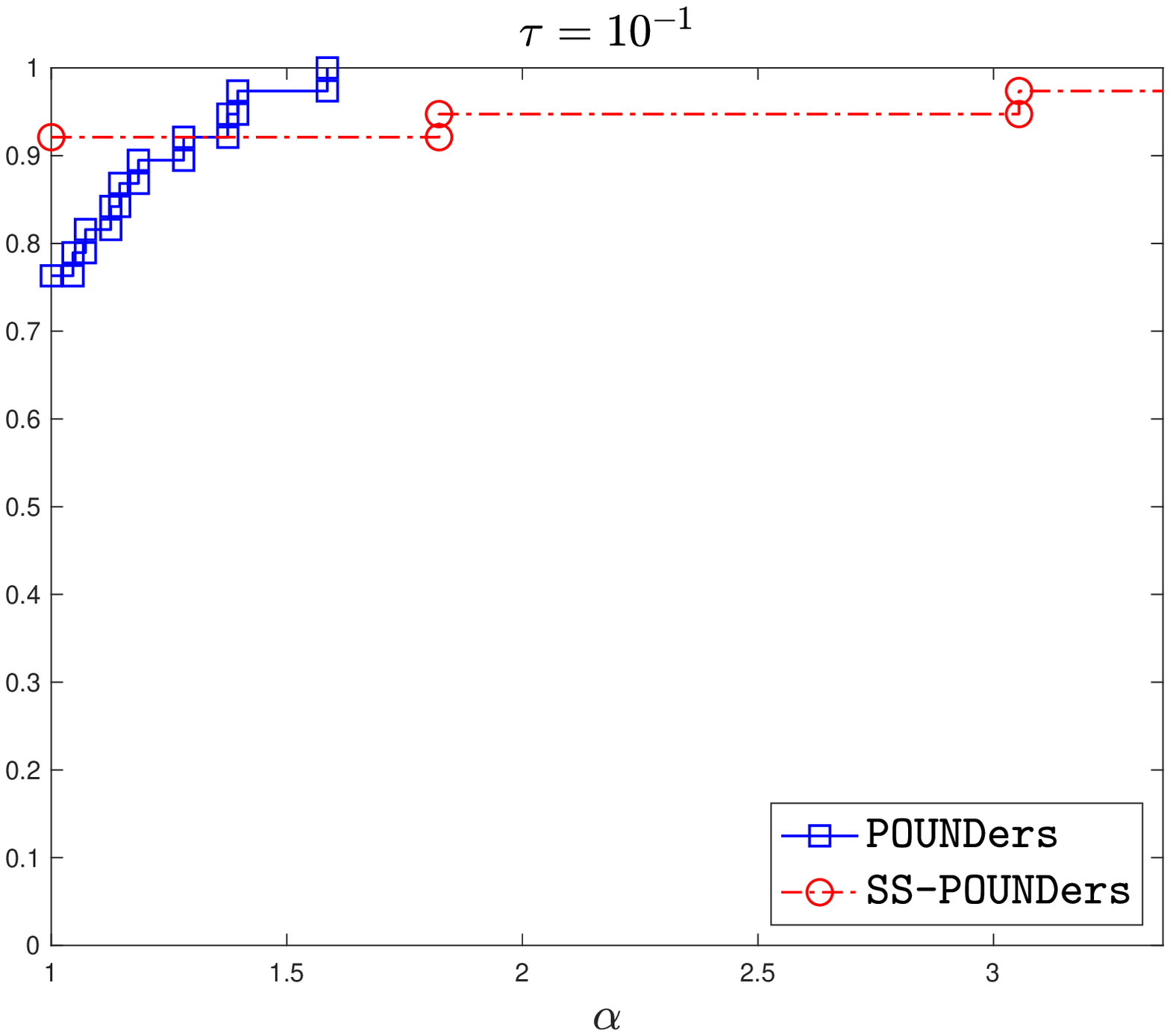}
    \includegraphics[width=.3\linewidth]{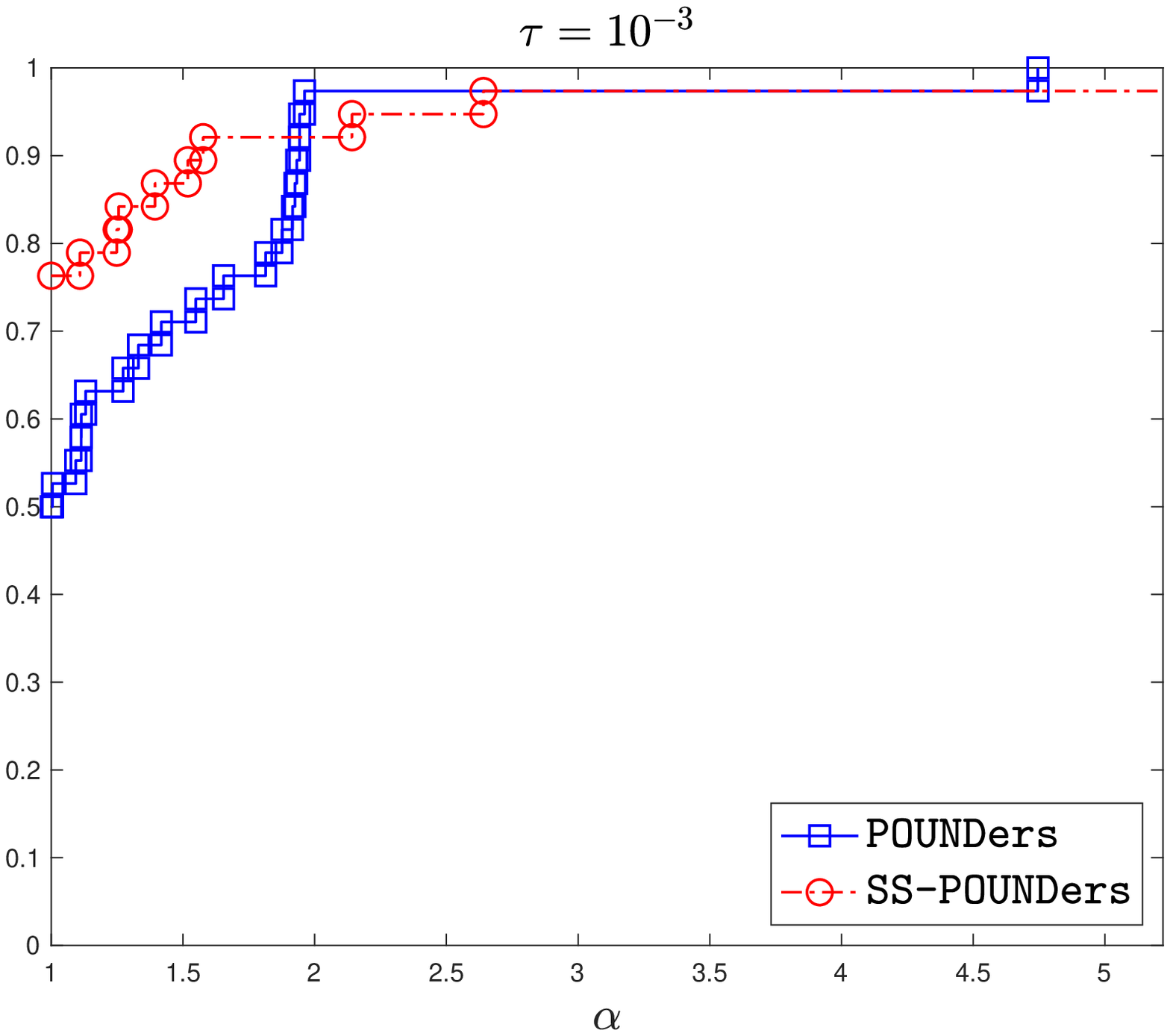}
    \includegraphics[width=.3\linewidth]{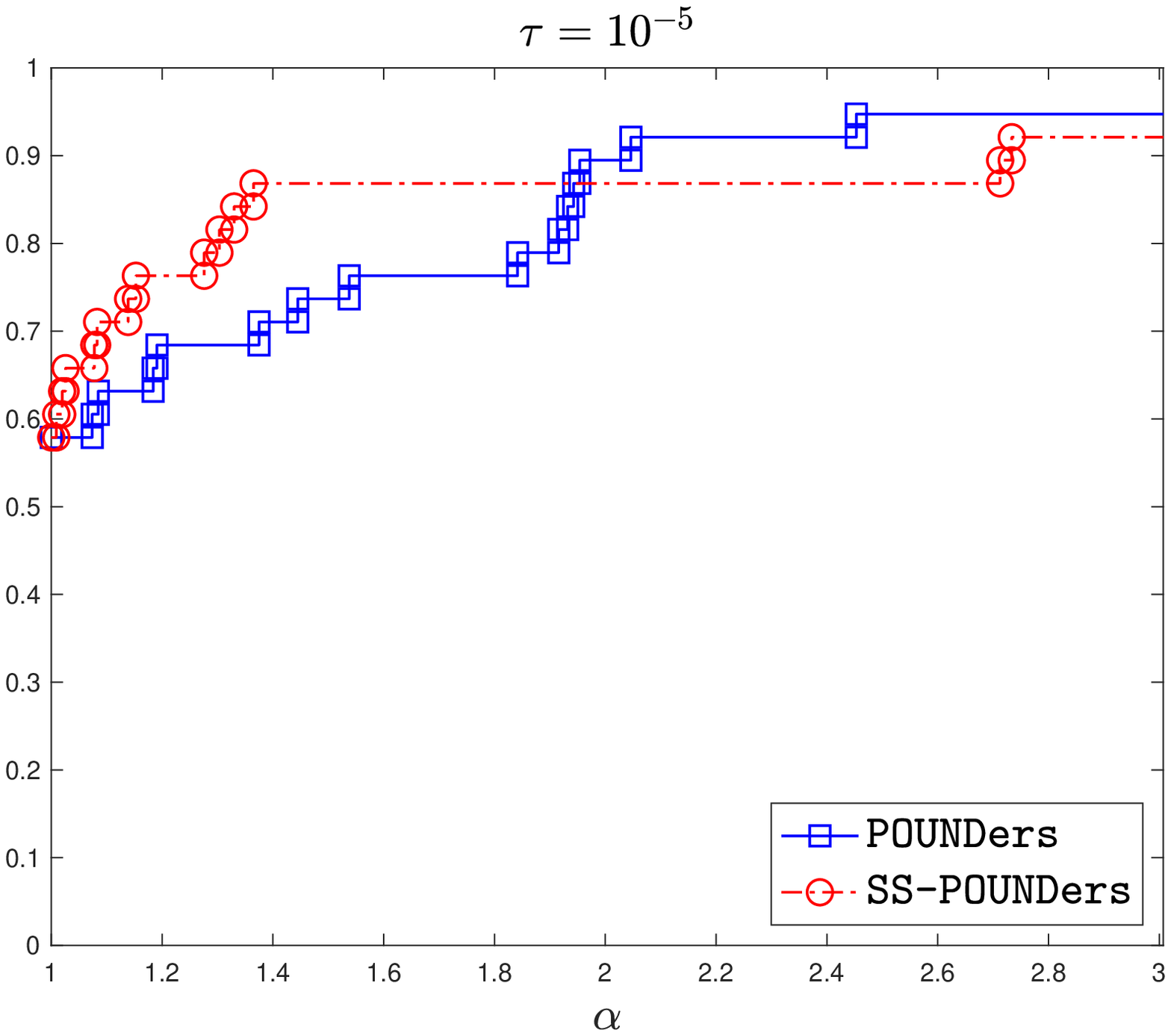}
\end{figure}

The results in \Cref{fig:worstbendfo} show that, in this lower-dimensional setting of BenDFO problems, randomization can, \emph{in the worst case}, harm the method fairly considerably. with the majority of problems being solved up to 5 times slower by \texttt{SS-POUNDers} than by \texttt{POUNDers}. 
However, in line with our previous intuition, on the higher dimensional problems, we see that for tighter convergence tolerances $(\tau\in\{10^{-3},10^{-5}\})$ there remains a compelling argument for employing \texttt{SS-POUNDers} over \texttt{POUNDers} across most of the test set, \emph{even in the worst case}. 
These are encouraging results, but suggest further empirical studies to better a priori quantify the dimension $n$ at which \texttt{SS-POUNDers} is, with high probability, more efficient in its use of function evaluations than \texttt{POUNDers}. 

\section{Conclusion}
In this paper, we presented a novel randomized algorithm for DFO, which we named \emph{basis sketching}. 
The algorithm employs a running average estimator of an interpolation model gradient and Hessian, as well as an unbiased counterpart, the ameliorated estimator. 
By using these estimators, basis sketching does \emph{not} require the maintenance of a fully linear model in every iteration of a model-based trust region method.
Using \texttt{POUNDers} as a starting point, we implemented a basis sketching method called \texttt{SS-POUNDers}, which never employs geometry improvement to yield fully linear models, but instead only ensures fully linearity restricted to randomized subspaces realized from judiciously chosen probability distributions. 
Our numerical results demonstrated, on problems with roughly $n=100$ variables, a fairly clear preference for using \texttt{SS-POUNDers} over \texttt{POUNDers}. 

This work was supported in part by the U.S.~Department of Energy, Office of Science, Office of Advanced Scientific Computing Research Applied Mathematics and SciDAC programs under Contract Nos.~DE-AC02-06CH11357 and DE-AC02-05CH11231.

\appendix
\section{Proof of \Cref{thm:sfullylinear}}\label{sec:proof1}
\begin{proof}
    For simplicity of notation, let 
    $$e^m(d) = m(y^0 + S^\top d) - f(y^0 + S^\top d) \quad
    \text{ and } \quad 
    e^g(d) = S\nabla m(y^0 + S^\top d) - S\nabla f(y^0 + S^\top d)$$
    for $d\in\Reals^p$.
    Without loss of generality, shift all points in $Y$ by $y^0$ so that $y^0 = 0$ (and hence, $\{y^0,y^1,\dots,y^p\} \subset \cB(0_p; c\Delta)$, 
    and each column of $\tilde{Y}$ is $Sy^i$). 
    Towards proving a result about $S$-full linearity, 
    suppose that $d$ satisfies $\|S^\top d\|\leq c\Delta$. 
    By \Cref{ass:cd} and our assumption on $m$, we may take a first-order Taylor expansion of $f$ about $y^0$, and so for $i=0,1,\dots,p$ we have

    \begin{equation}\label{eq:intermediate_eq1}
    \begin{array}{rl}
    \langle e^g(d), Sy^i - d\rangle & = 
    \displaystyle\int_0^1 \langle S\nabla f(y^0 + S^\top d + t(y^i - S^\top d)) - S\nabla f(y^0 + S^\top d), Sy^i - d\rangle \mathrm{d}t \\
    & - \displaystyle\int_0^1 \langle S\nabla m(y^0 + S^\top d + t(y^i - S^\top d)) - S\nabla m(y^0 + S^\top d), Sy^i -  d\rangle \mathrm{d}t\\
    & - e^m(d). 
    \end{array}
    \end{equation}
    Subtracting \eqref{eq:intermediate_eq1} with $i=0$ from each of \eqref{eq:intermediate_eq1} with $i=1,\dots,p$, we have
    \begin{equation}\label{eq:intermediate_eq2}
    \begin{array}{rl}
        \langle e^g(d), Sy^i\rangle = &
        \displaystyle\int_0^1 \langle S\nabla f(y^0 + S^\top d + t(y^i - S^\top d)) - S\nabla f(y^0 + S^\top d), Sy^i - d\rangle \mathrm{d}t \\
    & - \displaystyle\int_0^1 \langle S\nabla m(y^0 + S^\top d + t(y^i - S^\top d)) - S\nabla m(y^0 + S^\top d), Sy^i - d\rangle \mathrm{d}t\\
    & - \displaystyle\int_0^1 \langle S\nabla f(y^0 + (1-t)S^\top d) - S\nabla f(y^0 + S^\top d), -d\rangle\mathrm{d}t \\
    & + \displaystyle\int_0^1 \langle S\nabla m(y^0 + (1-t)S^\top d) - S\nabla m(y^0 + S^\top d), -d\rangle\mathrm{d}t
    \end{array}
    \end{equation}
    for $i=1,\dots,p$. 
    We now bound each row of the right hand side of \eqref{eq:intermediate_eq2} and will combine the bounds by Cauchy-Schwarz inequality. 
    The first row can be bounded as 
    $$
    \begin{array}{rl}
    & \displaystyle\int_0^1 \langle S\nabla f(y^0 + S^\top d + t(y^i - S^\top d)) - S\nabla f(y^0 + S^\top d), Sy^i -  d\rangle \mathrm{d}t\\
    = & \displaystyle\int_0^1 \langle \nabla f(y^0 + S^\top d + t(y^i - S^\top d)) - \nabla f(y^0 + S^\top d), S^\top Sy^i -  S^\top d\rangle \mathrm{d}t\\
    = & \displaystyle\int_0^1 \langle \nabla f(y^0 + S^\top d + t(y^i - S^\top d)) - \nabla f(y^0 + S^\top d), y^i -  S^\top d\rangle \mathrm{d}t\\
    \leq & \displaystyle\int_0^1 \|\nabla f(y^0 + S^\top d + t(y^i - S^\top d)) - \nabla f(y^0 + S^\top d)\|\|y^i - S^\top d\|\mathrm{d}t \\
    \leq & \displaystyle\int_0^1 L_g \|y^i - S^\top d\|^2 t \mathrm{d}t 
    = \frac{1}{2} L_g \|y^i - S^\top d\|^2 \leq 2L_g c^2\Delta^2,
    \end{array}
    $$
    where the second equality comes from the mutual orthonormality of the rows of $S$ and the fact that $y^i = S^\top \delta^i$ for some $\delta^i$, and the last line comes from the Lipschitz continuity of \Cref{ass:cd} and the fact that $\|y^i-S^\top d\|\leq \|y^i\| + \|S^\top d\| \leq 2c\Delta$.
    Bounding the remaining three lines of the right hand side of \eqref{eq:intermediate_eq2} is similar; 
    we arrive at a final bound of 
    $$\langle e^g(d), Sy^i\rangle \leq \frac{5}{2}c^2(L_g + L_{mg})\Delta^2.$$
    Using our matrix notation, the left hand side of these $p$ equations can be written
    $\tilde{Y}^\top e^g(d)$, and so we have the trivial bound
    $$\|\tilde{Y}^\top e^g(d)\| \leq \sqrt{p}\frac{5}{2}c^2(L_g + L_{mg})\Delta^2.$$
    By our supposition on $\|\tilde{Y}^{-1}\|$, 
    $$\|e^g(d)\| = \|\tilde{Y}^{-1}\tilde{Y}e^g(d)\| \leq \|\tilde{Y}^{-1}\|\|\tilde{Y}e^g(d)\| \leq \sqrt{p}\frac{5\Lambda}{2}c(L_g + L_{mg})\Delta,$$
    which is the value of $\kappa_{eg}$ that we intended to show. 

    To derive the value of $\kappa_{ef}$, 
    we return to \eqref{eq:intermediate_eq1} for $i=0$ and note that
    $$
    %\begin{array}{rl}
    e^m(d) 
    = %&  
    \displaystyle\int_0^1 \langle S\nabla f( (1-t) S^\top d) - S\nabla f(S^\top d),  - d\rangle \mathrm{d}t %\\
    %& 
    - \displaystyle\int_0^1 \langle S\nabla m((1-t)S^\top d) - S\nabla m( S^\top d), -  d\rangle \mathrm{d}t%\\
    %& 
    + \langle e^g(d), d\rangle,%\\
    %\end{array}
    $$
    and so by similar reasoning to our previous derivations,
    $$\begin{array}{rl}
    |e^m(d)| \leq 2(L_g + L_{mg})c^2\Delta^2 + \|e^g(d)\|\|d\| & \leq 2(L_g + L_{mg})c^2\Delta^2 + \sqrt{p}\frac{5\Lambda}{2}c^2(L_g + L_{mg})\Delta^2.\\
    & = \displaystyle\frac{4 + 5\Lambda\sqrt{p}}{2}(L_g + L_{mg})c^2\Delta^2,\\
    \end{array}
    $$
    as we meant to show. 
\end{proof}

\section{\Cref{alg:interpolation_set}}

\begin{algorithm}[h!]
    \caption{Determine Interpolation Set $Y$}
    \label{alg:interpolation_set}
    \textbf{Input:} Subspaces $S$ and $S^\perp$, initial set of points $Z = \{z^1=x, z^2, \dots, z^{rank(S) + 1}\}\subset\Reals^n$, bank of evaluated points $\cY = \{(y^1,f(y^1),\dots,(y^{|\cY|},f(y^{|\cY|}))\}$ satisfying $(z^i, f(z^i))\in\cY$ for all $i=1,\dots,|Z|$, trust region radius $\Delta$.\\
    \textbf{Initialize: } Choose algorithmic constants $c\geq 1$, $\theta_2 > 0$.\\
    Compute QR decomposition of $M(\Phi_S,Z)$, initialize $N = \emptyset$.\\
    \For{$i=1,\dots,|\cY|$}{
        \If{$y^i\not\in Z$ \text{ and } $\|y^i-x\|\leq c\Delta$}
        {
            Compute $N_+$ as in \eqref{eq:nplus}.\\
            \If{$\sigma_{\min}(N_+) \geq\theta_2$}{
                $Z \gets Z\cup \{y^i\}$\\
                $N \gets N_+$\\
            }
        }
    }
\end{algorithm}
\Cref{alg:interpolation_set} effectively operates by considering the effect of adding points to a partitioning of the Vandermonde matrix employed in the constraint of \eqref{eq:bask}.
We note that up to notation and subspace considerations, the statement and subsequent development of \Cref{alg:interpolation_set} is a very close analogue of \cite{SW08}[Algorithm 4.2].  

To ease our notation, for a given pair of orthogonal subspaces $S,S^\perp$, and for a given interpolation set $Y$ of size $p$, we will slightly abuse previous notation and abbreviate
$$M_S(Y) = M(\Phi_S,Y) \quad \text{ and } \msperp(Y) = M(\Phi_{S^\perp} \cup \{\Phi_Q\setminus\Phi_L\}, Y).$$
We can easily see that the stationarity and primal feasibility conditions from the KKT conditions of \eqref{eq:bask} can be arranged as the saddle-point system
\begin{equation}
    \label{eq:kkt_bask}
    \left[
    \begin{array}{cc}
     -\msperp(Y)\msperp(Y)^\top & M_S(Y)\\
     M_S(Y)^\top & 0\\
    \end{array}
    \right]
    \left[
    \begin{array}{c}
    \lambda\\
    \alpha\\
    \end{array}
    \right]
    =
    \left[
    \begin{array}{c}
    f(Y)\\
    0\\
    \end{array}
    \right],
\end{equation}
%where in the process we made the substitution $\beta = -\msperp(Y)^\top\lambda$. 
Let $N$ denote an orthogonal basis for the null space $\mathcal{N}(M_S(Y)^\top)$
and let $QR = M_S(Y)$ be a QR factorization. 
Because the second row of \eqref{eq:kkt_bask} indicates that $\lambda\in\mathcal{N}(M_S(Y)^\top)$, 
we can write $\lambda = N\omega$ for $\omega\in\Reals^{p-rank(S)-1}$ and so \eqref{eq:kkt_bask} reduces to $p$ equations
\begin{equation}
\label{eq:kkt_bask2}
\begin{array}{rcl}
    N^\top \msperp(Y)\msperp(Y)^\top N\omega & = & N^\top f(Y)\\
    R\alpha & = & Q^\top(f(Y) - \msperp(Y)\msperp(Y)^\top N\omega).\\
    \end{array}
\end{equation}
%As an aside, we note we can now express $\beta = $\msperp(Y)^\top N\omega$ from the solution of \eqref{eq:kkt_bask2}. 
We can now state and prove a very close analogue of \cite{SW08}[Theorem 3.2].
\begin{theorem}
    \label{thm:want_pd}
    Suppose $s := rank(S) \geq 2$. If both
    \begin{itemize}
        \item $rank(M_S(Y)^\top) = s + 1$, and
        \item $N^\top\msperp(Y)\msperp(Y)^\top N$ is positive definite,
    \end{itemize}
    then, for any $f(Y)\in\Reals^p$, there exists a unique solution $(\alpha^*, \beta^*, \gamma^*)$ to \eqref{eq:bask}. 
\end{theorem}

\begin{proof}
    Let both suppositions hold. 
    Because $N^\top\msperp(Y)\msperp(Y)^\top N$ is positive definite, we have that $\msperp(Y)^\top N$ is full rank. 
    Because $s\geq 2$, $\msperp(Y)^\top N$ being full rank in turn implies that its nullspace satisfies $\mathcal{N}(\msperp(Y)^\top N) = \{0_{p-s-1}\}$.
    Since the columns of $N$ form a basis for $\mathcal{N}(M_S(Y)^\top)$, 
    this implies that $\mathcal{N}(\msperp(Y)^\top)\cap \mathcal{N}(M_S(Y)^\top) = \{0_{p-s-1}\}$,
    which means that the full Vandermonde matrix $M(\Phi, Y) = [M_S(Y) \hspace{0.5pc} \msperp(Y)]$ is full rank. 
    Thus, the feasible region of \eqref{eq:bask} is nonempty. 

    Because the objective of \eqref{eq:bask} is convex in $\beta$ and $\gamma$, both he optimal solution $(\alpha^*,\beta^*,\gamma^*)$ to \eqref{eq:bask} and the Lagrange multipliers $\lambda$ in \eqref{eq:kkt_bask} are unique, as we meant to show. 
\end{proof}

With \Cref{thm:want_pd} in hand, we see that the intention of \Cref{alg:interpolation_set} is two-fold; it is designed to ensure that the second condition of \Cref{thm:want_pd} holds, 
while maintaining reasonable conditioning of the linear system in \eqref{eq:kkt_bask2}.
\Cref{alg:interpolation_set} computes an initial QR decomposition of the square matrix $M_S(Z)$ and maintains in the matrix $N$ an orthogonal basis for $\mathcal{N}(M_S(Z)^\top)$.
In a greedy fashion, if a point $y^i$ in the bank $\cY$ is not already in the interpolation set $Z$ and is sufficiently close to $x$, we consider the effect of adding it to $Z$.
By performing $s+1$ many Givens rotations to the initial QR factorization of $M(\Phi_S,Y)$,
we obtain an orthogonal basis for $\mathcal{N}(M_S(Y \cup \{y^i\})^\top)$ of the form 
$$N_+ = \left[
\begin{array}{cc}
N & Q\phi_1\\
0 & \phi_2\\
\end{array}
\right].$$
Thus, we can easily update
\begin{equation}
\label{eq:nplus}
    \msperp(Y \cup \{y^i\})^\top N_+ = \left[\msperp(Y)^\top N, \hspace{1pc} \msperp(Y)^\top Q\phi_1 + \phi_2\msperp(\{y^i\})^\top \right]
\end{equation}
Thus, by ensuring the least eigenvalue of $N_+$ is bounded away from 0 via the algorithmic parameter $\theta_2$, the output of \Cref{alg:interpolation_set} will guarantee that the second condition of \Cref{thm:want_pd} is met.
We note that the first condition is automatically met within the scope of \Cref{alg:bask}, since the initial set $Z$ will already contain the affinely independent points selected by \Cref{alg:subspace}, and adding to $Z$ cannot change the subspace that they span.  

\bibliography{bibs/smw-bigrefs.bib}
\bibliographystyle{abbrv} 

\framebox{\parbox{.90\linewidth}{\scriptsize The submitted manuscript has been created by
        UChicago Argonne, LLC, Operator of Argonne National Laboratory (``Argonne'').
        Argonne, a U.S.\ Department of Energy Office of Science laboratory, is operated
        under Contract No.\ DE-AC02-06CH11357.  The U.S.\ Government retains for itself,
        and others acting on its behalf, a paid-up nonexclusive, irrevocable worldwide
        license in said article to reproduce, prepare derivative works, distribute
        copies to the public, and perform publicly and display publicly, by or on
        behalf of the Government.  The Department of Energy will provide public access
        to these results of federally sponsored research in accordance with the DOE
        Public Access Plan \url{http://energy.gov/downloads/doe-public-access-plan}.}}

\end{document}